\documentclass[preprint]{elsarticle}
\usepackage{amsfonts,amssymb,mathrsfs,amsmath,amscd,epsfig,%setspace,
amstext,verbatim}

\newenvironment{proof}{\medskip                    %% Proof
\noindent{\scshape Proof:}}{\quad $\square$
\medskip}  %%

\usepackage{color}
\newtheorem{theorem}{Theorem}[section]

\newtheorem{proposition}{Proposition}[section]
\newtheorem{corollary}{Corollary}[section]
\newtheorem{example}{Example}[section]

\usepackage{graphicx,color,graphics}
\usepackage{latexsym}

\def\Sat{\operatorname{Sat}}

\def\digr{{\mathcal D}}

\def\B{\mathbb{B}}

\def\tplus{\oplus}

\begin{document}

\title{$(K,L)$-eigenvectors in max-min algebra}

\author[rvt]{Martin Gavalec\fnref{fn1}}
\ead{martin.gavalec@uhk.cz}

\author[rvt]{Zuzana N\v{e}mcov\'{a}\fnref{fn1}}
\ead{zuzana.nemcova@uhk.cz}

\author[rvt2]{Serge{\u\i} Sergeev\corref{cor}\fnref{fn2}}
\ead{sergiej@gmail.com}

\address[rvt]{University of Hradec Kr\'{a}lov\'{e}/Faculty of Informatics and Management, Rokitansk\'{e}ho 62,
Hradec Kr\'{a}lov\'{e} 3, 500 03, Czech Republic}
\address[rvt2]{University of Birmingham, School of Mathematics, Edgbaston B15 2TT}

\cortext[cor]{Corresponding author. Email: sergiej@gmail.com}
\fntext[fn1]{Supported by the Czech Science Foundation grant \#18-01246S}

\fntext[fn2]{Supported by EPSRC grant EP/P019676/1}

%%%%%%%%%%%%%%%%%%%%%%%%%%%%%%%%%%%%%%%%%%%%%%%%%%%%%%%%%%%%%%%%%%%%%%%%
%%%                         ABSTRACT
%%%%%%%%%%%%%%%%%%%%%%%%%%%%%%%%%%%%%%%%%%%%%%%%%%%%%%%%%%%%%%%%%%%%%%%%
\begin{abstract}
Using the concept of $(K,L)$-eigenvector, we investigate the structure of the max-min eigenspace associated with a given eigenvalue of a matrix in the max-min algebra (also known as fuzzy algebra). In our approach, the max-min eigenspace is split into several regions according to the order relations between the eigenvalue
and the components of $x$. The resulting theory of $(K,L)$-eigenvectors, being based on
the fundamental results of Gondran and Minoux, allows to describe the whole max-min eigenspace
explicitly and in more detail.
\end{abstract}

%%%%%%%%%%%%%%%%%%%%%%%%%%%%%%%%%%%%%%%%%%%%%%%%%%%%%%%%%%%%%%%%%%%%%%%%
%%%                         KEYWORDS
%%%%%%%%%%%%%%%%%%%%%%%%%%%%%%%%%%%%%%%%%%%%%%%%%%%%%%%%%%%%%%%%%%%%%%%%

\maketitle

\section{Introduction}\label{s:introduction}
%---------------------------------------------------------------------------------------------

By max-min algebra we mean the unit interval $\B=<0,1>$ equipped with the
arithmetic operations of "addition" $a\oplus b=\max(a,b)$ and
"multiplication" $a\otimes b=\min(a,b)$. Algebraically speaking, max-min algebra is a semiring where
both arithmetic operations are idempotent. Let us also note that, algebraically, max-min algebra is an example of
incline algebra of~\cite{CKR:84}.

The arithmetic operations $\oplus$ and $\otimes$ can be extended
to matrices and vectors in the usual way so that for any matrices $A=(a_{ij})$ and
$B=(b_{ij})$ of appropriate dimensions we can define their "sum" $A\oplus B$ and
"product" $A\otimes B$ by the usual rules:
$(A\oplus B)_{ij}=a_{ij}\oplus b_{ij}$ and $(A\otimes B)_{ij}=\bigoplus_{k} a_{ik}\otimes b_{kj}$.
For a square matrix $A\in\B^{n\times n}$ we can also define its max-min matrix powers: $A^k=\underbrace{A\otimes\ldots\otimes A}_k$, where $k$ is a natural number. Note that $A^0=I$, the usual identity matrix.
Further we will systematically omit the product sign $\otimes$, for brevity.

As usual in tropical/idempotent algebras, to each matrix $A\in\B^{n\times n}$ we
associate a digraph $\digr(A)=(N,E)$ with the node set $N=\{1,\ldots,n\}$ and edge set
$E=\{(i,j)\colon a_{ij}>0\}$. Each edge has a weight $a_{ij}$.
A sequence of edges $(i_0,i_1),(i_1,i_2),\ldots, (i_{k-1},i_k)$ where each edge belongs
to $E$ is called a walk whose length is $k$ and whose weight is given by the max-min product
$a_{i_0i_1}\dots a_{i_{k-1}i_k}$.

One of the motivations to study max-min algebra comes from the theory of fuzzy sets where the operation
$\min(a,b)$ is one of the most useful examples of triangular norms, see Klement, Mesiar and Pap~\cite{KMP:00}. See also Gondran and Minoux~\cite{GM-07} for more on semirings and other algebraic models relevant to the theory of fuzzy sets.

\if{
It is easy to see that each entry $(A^k)_{ij}$ is the greatest max-min weight of a walk connecting
node $i$ to node $j$ and having length $k$. This gives rise to the following example of algebraic optimal path problem~\cite{Car-71,LM-98} to which the theory of max-min matrix powers can be applied. Suppose that
we are given a network of roads with one bridge on each road. The capacity of each bridge
(i.e., the biggest weight of a vehicle that can go over it) is given. The capacity of a sequence of such roads with bridges (i.e., a walk in this network) is then determined as the minimal capacity of the bridges in it.
The problem then is, given a starting point and an end point and (possibly) the number of bridges to be
passed, to find a sequence of roads with the greatest capacity.
}\fi

%The main goal of the present paper is to develop a new approach to the max-min eigenproblem.
The main goal of the present paper is to further investigate the structure of the max-min eigenspace associated with a given eigenvalue. 
For a given matrix $A\in\B^{n\times n}$ and a number $\lambda\in\B$, the max-min eigenspace associated with $\lambda$ is
the set of vectors $x\in\B^{n\times 1}$ (called $\lambda$-eigenvectors) such that
\begin{equation}
\label{e:eig}
A  x=\lambda  x.
\end{equation}
%The set of these vectors forms a max-min $\lambda$-eigenspace of $A$. 
Note that it is indeed
a space in the sense of max-min algebra: for any $\alpha,\beta\in\B$ and any $u$ and $v$
satisfying~\eqref{e:eig} $\alpha u\oplus\beta v$ also satisfies this equation. Let us now give some examples to motivate the study of this eigenproblem and our approach to it.

\begin{example}[Medical symptoms and therapy]
{\rm One of the first models, in which the max-min eigenvectors were used, was concerned with searching invariants in therapeutic recommendations. It was put forward by Sanchez~\cite{San-78}, see also a monograph by Rakus-Andersson~\cite{Rak:07} for a more recent account of this method and more references. In this model, the matrix $A\in\B^{n\times n}$  expresses a fuzzy relation describing the relative success of treating $n$ symptoms of some illness by some drug. In this matrix, a diagonal entry $a_{ii}$ is equal to the recovery rate from the $i$th symptom, and entry $a_{ij}$, for any $i\neq j$ is, the rate at which either both symptoms $i$ and $j$ are absent after the application of drug, or symptom $i$ is absent and symptom $j$ is still present. In other words, $a_{ij}$ is the rate at which ``the action of the drug is equal or stronger on the $i$th symptom than on the $j$th symptom''~\cite{Rak:07}, designed to account for the link between the symptoms and action of the drug on them.
The max-min eigenvectors $x$ associated with eigenvalue $1$ appear as {\em eigen fuzzy sets}. Note that in this application, one is usually interested in the greatest left eigenvector of $A$, i.e., the greatest $x$ such that $xA=x$. This vector, together with some other vector, is then used to estimate reliable intervals for the rate at which a drug removes the $n$ symptoms. 

The greatest max-min eigenvector is easy to use, as it can be quickly found by a simple iterative procedure~\cite{San-81,Rak:07}. However, it can be seen as giving an overly optimistic judgement, and in this situation one may be interested  1) to impose a condition that the recovery rate is no bigger than $\lambda$ and 2) to require that certain components of vector $x$, while being below $\lambda$, stay at their initial level $x_i$. A further development of this idea in the context of medical treatments is beyond the scope of the present paper.} 
\end{example}

The greatest left max-min eigenvectors of a given matrix $A$ with entries in the interval $[0,1]$ were also used by Nobuhara et al.~\cite{NBH-06} in image reconstruction, where a given image was encoded by means of max-min and min-max eigenvectors of $A$, as well as fixed points of various convex combinations of the max-min and min-max products associated with $A$, and then successfully reconstructed using them.

\begin{example}[Security in a computer framework]
{\rm This example was put forward as a motivation for studying strongly tolerant interval eigenvectors in max-min algebra~\cite{GPP-19}. A computer network consisting of servers $S_1,\ldots, S_n$, data storage units $D_1,\ldots, D_n$ and a logical unit $L$  is considered. Lines $K_{ij}$, for $i,j\in I$, connect $S_i$ with $D_j$, while lines $L_j$ connect every $D_j$ with $L$. The security level of each line in the network is measured by values in the real interval $[0,1]$, where value $0$ stands for completely unsecure connection and value $1$ stands for completely secure connection. The security of 
every $K_{ij}$ (data security) is denoted by $a_{ij}$, and the security of $L_j$ (logic security) is denoted by $x_j$. The maximal security level $y_i$ of information transfer from servers $S_i$ to the logical unit $L$ through the given storage units is the $i$th component of the vector $y=A\otimes x$. 
If we wish to keep the data security unchanged by  the information transfer from servers $S_i$ to the logical unit $L$ through the given storage units, then $y = x$ should hold. 
Moreover, we have to consider the security level of the system itself, which depends on the accessible technologies and the available budget. The technological security of the system can decrease the security of the processed data below the level $\lambda$, for some $\lambda\in \langle 0,1\rangle$, and then $\lambda x$ becomes new target level of security. Note that $A (\lambda x)=\lambda x$ in max-min algebra, which means that the invariance of the new levels of security will be ensured. If we additionally require that the new levels of security should become exactly equal to $\lambda$ in some components and stay equal to $x_i$ in the others, then we are led to study $(K,L)$- eigenvectors, introduced and studied in the present paper.}  
\end{example}

Problem~\eqref{e:eig} has been studied in max-min algebra at least since Gondran and Minoux~\cite{Gon-76,GM-78} and Sanchez~\cite{San-78}. The approach
taken in~\cite{Gon-76,GM-78} resembles that of max-plus algebra, where eigenspaces are characterized as particular subspaces of the column span of Kleene star. There is also a number of works where a different approach
is taken. Sanchez~\cite{San-78,San-81} is focused on the largest eigenvector associated with the eigenvalue $1$ and suggests an algorithm that can be used to compute it in practice.
In the same vein, Cechl\'arov\'a~\cite{cechlarova1992} describes lower and upper basic eigenvectors in terms of the associated graph.
The structure of the eigenspace of increasing max-min eigenvectors $x_1\leq x_2\leq\ldots\leq x_n$ in max-min algebra is described by Gavalec \cite{G-02}, and
various types of max-min interval eigenvectors have been studied in Gavalec et al.~\cite{GPT-14}.

Gondran and Minoux obtained fundamental results for~\eqref{e:eig} also over more general semirings
with idempotent multiplication, see~\cite{GM:08}[Section 6.3] for one of the latest accounts. We are going to
use these results.
However, we observe that the theory as presented in that monograph describes only the
solutions whose all components are less than or equal to $\lambda$ \cite{GM:08}[Ch. 6, Corollary 3.5].
Several examples when~\eqref{e:eig}
also admits other solutions can be found in the present paper. To describe those other solutions
we adopt an approach which is similar to that of Gavalec et al.~\cite{GNS-15}, where the eigenproblem in max-{\L}ukasiewicz algebra was studied (see also Gavalec and N\v{e}mcov\'a~\cite{GN-17}).  Namely, given $A$ and $\lambda$ we
consider a partition of $N=\{1,\ldots,n\}$ into two disjoint subsets $K$ and $L$ such that
$K\cup L=N$ and pose a problem of describing all $\lambda$-eigenvectors $x$ that satisfy $x_i\leq \lambda$ and
hence $\lambda x_i=x_i$ for all
$i\in K$ and $x_i\geq\lambda$ and hence $\lambda x_i=\lambda$ for all $i\in L$. When $K=N$ we call
such vectors "principal $\lambda-$eigenvectors" since~\eqref{e:eig} becomes $Ax=x$, and when $L=N$ we call such vectors
"background $\lambda$-eigenvectors", in analogy with~\cite{GNS-15}. In the latter case~\eqref{e:eig} becomes
$Ax=\lambda\mathbf{1}$, where $\mathbf{1}$ denotes the vector of all $1$'s. principal eigenvectors were described
in~\cite{GM:08}[Ch. 6, Corollary 3.5], which we revisit here in Corollary~\ref{c:principal}. Background
eigenvectors are easy to obtain: see Proposition~\ref{p:backgr} and Proposition~\ref{p:lambdaboxgen} below.
Principal and background $\lambda$-eigenvectors are fundamental for describing the $(K,L)-$eigenvectors associated with $\lambda$ in the case of general $K$, $L$ and $\lambda$. Their description is stated in Theorem~\ref{t:mainres}, which can be considered as our main result.

Theorem~\ref{t:mainres} also yields a method for constructing a generating set for the whole $\lambda$-eigenspace, although the computational effort may grow exponentially with matrix size.
First of all, note that, for general $K$ and $L$, the set of $(K,L)-$eigenvectors is not a max-min space any more, since the set $\{x\colon x_i\geq \lambda\}$ is not a max-min space. However, it is a {\em max-min convex set}: if $x$ and $y$ are $(K,L)$-eigenvectors associated with $\lambda$ and $\alpha\oplus\beta=1,$ then $\alpha x\oplus \beta y$ is also such a $(K,L)$-eigenvector. This follows since (like in the usual convexity) any max-min space is a max-min convex set, and the sets defined by max-min affine inequalities such as $x_i\leq \lambda$ or $x_i\geq\lambda$ are also max-min convex. For arbitrary vectors $x^1,\ldots, x^n\in\B^n$ one can consider their {\em max-min convex combinations}: $\alpha_1 x^1\oplus\ldots\oplus \alpha_n x^n$ where $\alpha_1\oplus\ldots\oplus \alpha_n=1$ is required, and the set of all such combinations (for given $x^1,\ldots, x^n$) is called the {\em max-min convex hull} of these points.
Some max-min convex sets are max-min convex hulls of a finite number of points. The easiest and most important example is $\B^n$, which is the max-min convex hull of the zero vector $\mathbf{0}$ and unit vectors $e^i$, whose $i$the component is equal to $1$ and the rest of the components are equal to $0$. 
Theorem~\ref{t:mainres} will provide for a comprehensive description of any set of $(K,L)$-eigenvectors as a max-min convex hull of a finite number of points. Then we can take the union of these sets and obtain a set of points whose max-min convex hull is the whole $\lambda$-eigenspace. Taking out the zero vector $\mathbf{0}$, we then obtain a generating set for the whole 
$\lambda$-eigenspace. See, e.g., Nitica and Sergeev~\cite{NS-14,NS-15} for more on max-min convexity. 

All new results of this paper are obtained in Section~\ref{s:mainres}.
Preliminary notions and results from Gondran and Minoux~\cite{GM:08} and Butkovi\v{c} et al.~\cite{BSS-12}, which provide the necessary algebraic tools, are given in Section~\ref{s:tropical}.

\section{Some problems of max-min algebra}
\label{s:tropical}
%---------------------------------------------------------------------------------------------

In this section we will give some necessary notions and facts from max-min algebra
on which our study of max-min $(K,L)$-eigenvectors will be based.
We will start with defining the notions of metric matrix and Kleene star and
(following Gondran and Minoux~\cite{GM:08}) giving a description of the set of principal eigenvectors ($x$ such that
$Ax=x$). This will be followed by describing the solution set to max-min Bellman ($Z$-matrix) equation
(following Butkovi\v{c} et al. \cite{BSS-12} or Krivulin \cite{Kri:09}) and solving a special type of
max-min equation~\eqref{e:Axb} in Subsection~\ref{s:special}.

\subsection{Metric matrix, Kleene star and the principal eigenproblem}
%---------------------------------------------------------------------------------------------
For a square matrix $A\in\B^{n\times n}$ let us define its metric matrix $A^+=(a^+_{ij})_{i,j=1}^n$ and
Kleene star $(a^*_{ij})_{i,j=1}^n$ by the following series:
\begin{equation}
\begin{split}
A^+&=A\oplus A^2\oplus A^3\oplus\ldots\\
A^*&=I\oplus A\oplus A^2\oplus\ldots
\end{split}
\end{equation}
where $I$ denotes the identity matrix, whose diagonal entries are equal to $1$ and off-diagonal 
entries are $0$.
It is well-known that in max-min algebra these series always converge and, moreover, can be truncated:
\begin{equation}
\begin{split}
A^+&=A\oplus \ldots\oplus A^n,\\
A^*&=I\oplus A\oplus\ldots\oplus A^{n-1}.
\end{split}
\end{equation}
%
%The following properties of metric matrix and Kleene star are well-known,
%see~\cite{GM:08}[Chapter 6]:
%\begin{equation}
%\label{kls-prop}
%\begin{split}
%A^*&=AA^*\oplus I=A^*A\oplus I,\\
%A^+&=AA^*=A^*A.
%\end{split}
%\end{equation}
%
In terms of the associated graph, $a^+_{ij}$, being equal to $a^*_{ij}$ when $i\neq j$,
is the maximal (max-min) weight of all walks connecting $i$ to $j$ with unrestricted
length. 
%In terms of the network with bridges, this is the maximal capacity of
%all sequences of roads connecting $i$ to $j$. 
So is the {\em optimal walk interpretation}
of metric matrices and Kleene stars.

Metric matrices, Kleene stars and associated digraphs provide
some of the basic tools for the max-min eigenproblem.
Let us start with the {\em principal eigenproblem}: the problem of
identifying all vectors $x$ that satisfy $Ax=x$ for
a given matrix $A$. Such vectors are called {\em principal eigenvectors}
of $A$.

For each principal eigenvector $x$,
following the terminology of~\cite{BCOQ},
define its {\em saturation graph} $\Sat(A,x)$ as the graph
consisting of all edges $(i,j)$ that satisfy $a_{ij}\otimes x_j=x_i$ and all nodes on these edges. This graph in general has several maximal strongly connected components, and let $C(A,x)$ denote a subset of $N=\{1,\ldots,n\}$ that contains one node from each strongly connected component of $\Sat(A,x)$. We now state a description of the set of principal eigenvectors, which is %essentially
due to~\cite{GM:08}. %The proof is also given here, for convenience of the reader.
\begin{theorem}[Gondran-Minoux~\cite{GM:08} Section 6.3]
\label{t:principal}
The set of principal ei\-genvectors is a max-min space generated by
vectors
\begin{equation}
\label{e:genset-princ}
a_{ii}^+ (A^*)_{\cdot i},\quad i=1,\ldots,n.
\end{equation}
More precisely, each vector of~\eqref{e:genset-princ} is a principal eigenvector,
and each principal eigenvector $x$ can be represented as
\begin{equation}
\label{e:repr-princ}
x=\bigoplus_{i\in C(A,x)} x_i a_{ii}^+ (A^*)_{\cdot i}.
\end{equation}
\end{theorem}

\if{
\begin{proof}
We first show that each vector in~\eqref{e:genset-princ} is a principal eigenvector.
We have:
\begin{equation}
\label{e:cross-eig}
A(a_{ii}^+ (A^*)_{\cdot i})=a_{ii}^+ (A^+)_{\cdot i},
\end{equation}
and we need to show that $a_{ii}^+ (A^*)_{\cdot i}=a_{ii}^+ (A^+)_{\cdot i}$.
By their definition, matrices $A^*$ and $A^+$ differ only on the diagonal and therefore
$a_{ii}^+ a^*_{ki}=a_{ii}^+ a^+_{ki}$ for $k\neq i$, and for $k=i$ we have
$a_{ii}^+=(a_{ii}^+)^2$ by the idempotency of multiplication. This implies
$a_{ii}^+ (A^*)_{\cdot i}=a_{ii}^+ (A^+)_{\cdot i}$.

We now take an arbitrary principal eigenvector $x$. Since $Ax=x$ we also have
$A^kx=x$ for any $k\geq 1$ and therefore also $A^*x=x$, adding up all these equalities
and using the idempotency of $\oplus$.
Writing this in terms of columns of $A^*$ we have
\begin{equation}
\label{e:repr-simple}
x=\bigoplus_{i=1}^n x_i (A^*)_i.
\end{equation}
To show that this is the same as~\eqref{e:repr-princ} we first take $i\in C(A,x)$. Since $i$ is on
a cycle of $\Sat(A,x)$, for some $i_1,\ldots, i_k$ where $i=i_1$ we have that
$$a_{i_1i_2}x_{i_2}=x_{i_1},\, a_{i_2i_3}x_{i_3}=x_{i_2},\,\ldots,\,a_{i_ki_1}x_{i_1}=x_{i_k}.$$
These equations imply that
$$a_{i_1i_2}a_{i_2i_3}\cdot\ldots\cdot a_{i_ki_1}x_{i_1}=x_{i_1}$$
and hence $a_{ii}^+x_i\geq x_i$, which is equivalent to $a_{ii}^+ x_i= x_i$. Therefore
$x_i$ can be replaced in~\eqref{e:repr-simple} with $x_i a_{ii}^+$, for such $i$.

Since each node has an outgoing edge in $\Sat(A,x)$, for each $j\notin C(A,x)$ there exists a
walk in $\Sat(A,x)$ connecting it to a (maximal) strongly connected component of $\Sat(A,x)$
and hence a walk connecting to a node $i\in C(A,x)$, which is a sequence
$i_1,\ldots, i_k$ with $j=i_1$ and $i=i_k$ such that
$$a_{i_1i_2}x_{i_2}=x_{i_1},\, a_{i_2i_3}x_{i_3}=x_{i_2},\,\ldots,\,a_{i_{k-1}i_k}x_{i_k}=x_{i_{k-1}}.$$
These equations imply that
$$a_{i_1i_2}\cdot\ldots\cdot a_{i_{k-1}i_k}x_{i_k}=x_{i_1}.$$
Using this we obtain that for arbitrary index $\ell$
\begin{equation*}
\begin{split}
a^*_{\ell j}x_j=a^*_{\ell i_1}x_{i_1}=a^*_{\ell i_1} a_{i_1i_2}\cdot\ldots\cdot a_{i_{k-1}i_k}x_{i_k}\leq a^*_{\ell i_k} x_k=
a^*_{\ell i} x_i
\end{split}
\end{equation*}
where the last inequality is due to the optimal walk interpretation of Kleene star.
Hence we have shown that for each $j\notin C(A,x)$ there exists $i\in C(A,x)$ with
$a^*_{\ell j}x_j\leq a^*_{\ell i}x_i$. Hence we can omit the terms in~\eqref{e:repr-simple} where
$i\notin C(A,x)$, while the terms for which $i\in C(A,x)$ can be multiplied by $a_{ii}^+$.
This shows~\eqref{e:repr-princ} and hence the whole claim.
\end{proof}
}\fi

%\subsection{Bellman equation}
%---------------------------------------------------------------------------------------------
We will also use the theory of algebraic Bellman equation
\begin{equation}
\label{e:bellman}
x=Ax\tplus b,
\end{equation}
studied over general semirings, e.g., in Carr\'e \cite{Car-71}, Litvinov and Maslov \cite{LM-98}. In nonnegative linear algebra this equation is
also known as $Z$-matrix equation~\cite{BSS-12}.

Although~\eqref{e:bellman} has been known for decades, the following fundamental
result, describing the whole solution set to~\eqref{e:bellman}, 
was formulated only recently in Butkovi\v{c} et al. and Krivulin~\cite{BSS-12,Kri-06,Kri:09}.
%The solution set of~\eqref{e:bellman} is described as follows.

%will be denoted by $S(A,b)$ . A short proof that is
%close to the one given in~\cite{GNS-15} is presented for reader's convenience.

\begin{theorem}
\label{t:schneider}
Let $A\in\B^{n\times n}$ and $b\in\B^n$. Equation~\eqref{e:bellman}
always has 
solutions, and the set of these solutions is
\begin{equation}
\label{sab-descr}
\{A^*b\tplus v\colon Av=v\}.
\end{equation}
\end{theorem}
\if{
\begin{proof} First it can be verified that any vector like on the r.h.s.
of~\eqref{sab-descr} satisfies $x=Ax\tplus b$, using that
$A^*=AA^*\tplus I$ \eqref{kls-prop}. In particular, a solution of~\eqref{e:bellman} always
exists since $A^*$ always converges in the max-min case.

Iterating the equation $x=Ax\tplus b$ we obtain
\begin{equation*}
\begin{split}
x=Ax\tplus b= &A(Ax\tplus b)\tplus b=\ldots\\
=& A^k x \tplus (A^{k-1}\tplus\ldots\tplus I)b= A^k x\tplus A^*b
\end{split}
\end{equation*}
for all $k\geq n$.

This implies $x\geq A^*b$.  Further, $x$ satisfies $Ax\leq x$, hence
$x\geq Ax\geq\ldots\geq A^kx\geq\ldots$.

In max-min algebra, the orbit $\{A^kx\}_{k\geq 1}$ always starts to cycle
from some $k$. Indeed, in this algebra every entry of $A^kx$ for any $k$ is an entry of $A$ or an entry of $x$, and therefore there exist $k_1$ and $k_2$ for which $A^{k_1}x=A^{k_2}x$. But as $Ax\leq x$, we have $x\geq Ax\geq\ldots\geq
A^kx\geq\ldots$, and it is only possible that the sequence
$\{A^kx\}_{k\geq 1}$ stabilizes starting from some $k$. That is,
starting from some $k$, vector $v=A^k x$ satisfies $Av=v$. The proof
is complete.
\end{proof}
}\fi

\subsection{Special type of equation} \label{s:special}
%---------------------------------------------------------------------------------------------
We also need to describe the solution set for the system

\begin{equation}
\label{e:Axb}
A z \oplus b =\lambda \mathbf{1}.
\end{equation}
where $A\in\B^{m\times n}$, $z\in\B^{n}$ and $b\in\B^m$ (for arbitrary natural numbers $m$ and $n$).

%Below we describe the solution space for this system and the minimal set of its gnerators.
%Solutions $(z_{\Tilde{L}},z_K)$ are then to be substituted back in~\eqref{e:SKL-rep}.
%\newpage
We will study this system under the condition that all coefficients of $A$ and $b$ are less than or equal to $\lambda$:
\begin{equation}
\label{e:leqlambda}
a_{ij}\leq\lambda,\quad b_i\leq\lambda\quad\forall i,j\in N.
\end{equation}
The description will be obtained in terms of coverings and minimal coverings,
following the known solution method for systems of the form $A\otimes x=b$ and $A\otimes x\geq b$ in max-plus and
max-min algebra (see, e.g.,~Butkovi\v{c} \cite{But:10}, Cunninghame-Green and Cechl\'arov\'a \cite{CC-95}, and Elbassioni~\cite{Elb-08}).

Note that if the system~\eqref{e:Axb} is solvable,
then in every row of the equation, we have to obtain $\lambda$
either from $Az$ or from $b$.
\vspace{3pt}

Let us denote $I_0=\{i\colon b_i<\lambda\}$ and $C_j=\{i \in I_0 \colon a_{ij}=\lambda\}$ for $j\in N$. \\[3pt]
Furthermore, for $W\subseteq N$  we denote $C^W$ and  $z^W$   by putting
\begin{align}
C^W     &= \bigl\{  C_j \colon j\in W  \bigr\}\\[3pt]
z^W_j   &= \left\{
        \begin{array}{ll}
            \lambda \quad & \text{if } \ j \in W \\
            0 & \text{otherwise}
        \end{array}
    \right.
		\label{e:CWzW}
\end{align}
We say that $C^W$ is a {\em covering} of $I_0$, if $I_0 = \bigcup C^W(=\bigcup_{j\in W} C_j)$. Moreover,  $C^W$ is a {\em minimal covering} of $I_0$, if
\begin{equation}
\label{e:min_covering}
I_0 \ne \bigcup C^V\, \text{ for every }\, V\subseteq W,\, V\ne W.
\end{equation}
\begin{proposition}
\label{p:solution_xW}
$z\in\B^n$ is a solution of \eqref{e:Axb} with condition~\eqref{e:leqlambda}
if and only if $z\geq z^W$ for some $W\subseteq N$ such that  $C^W$ is a minimal covering of $I_0$.
\end{proposition}
\begin{proof}
Suppose  $z\in\B^n$ is a solution of \eqref{e:Axb}. Then $b_i< \lambda$ implies $(A_i)z = \lambda$.
Then, also using~\eqref{e:leqlambda}, we obtain that for every $i\in I_0$ there is $j=j(i)\in N$ with $a_{ij(i)}= \lambda$ and $z_{j(i)} \geq \lambda$. Denoting $W= \{j(i) \colon i\in I_0 \}$ , we get $I_0 = \bigcup C^W$ and $z\geq z^W$. Without loss of generality, we can assume that the covering $C^W$ of $I_0$ is a minimal one.

For the converse implication, suppose $z\geq z^W$ for some $W\subseteq N$, with $C^W$ being a minimal covering of $I_0$.
Then, for any $i\in I_0$, there is $j= j(i) \in W$ such that $i\in C_{j(i)}$. That is, $z_{j(i)}\geq z^W_{j(i)}= \lambda$ and $a_{ij(i)}= \lambda$, which gives $(A_i)z = \lambda$.
\end{proof}
\begin{corollary}
\label{c:Axb}
Equation~\eqref{e:Axb} with condition~\eqref{e:leqlambda} is solvable if and only if
there exists a covering $C^W$ of $I_0$.
In this case, solution set to~\eqref{e:Axb} can be represented as $\bigcup_W S^W$ where the union is taken over
$W\subseteq N$ such that $C^W$ is a minimal covering of $I_0$ and
\begin{equation}
\label{e:SW}
\begin{split}
S^W:&=\{z\colon z^W\leq z\leq\mathbf{1}\}\\
&=\{z\colon \lambda \leq z_k \leq 1 \ \text{for $k\in W$},\  z_k\geq 0\ \text{for
$k\in N \setminus W$}\},
\end{split}
\end{equation}
\end{corollary}
Note that if $I_0=\emptyset$ then the minimal covering of $I_0$ is $C^W$ with $W=\emptyset$ and
the unique minimal (and hence the least) solution of~\eqref{e:Axb} is $\mathbf{0}$ and
the solution set is $\{z\colon \mathbf{0}\leq z\leq\mathbf{1}\}$.

 Set $S^W=\{z\colon z^W\leq z\leq\mathbf{1}\}$
can be also algebraically expressed as follows.
\begin{equation}
\label{e:SWM}
S^W=\{z^W\oplus  v_{N}\colon v_{N}\in\B^{|N|}\}
\end{equation}

\begin{example}\normalfont
We shall illustrate the solution of \eqref{e:Axb} by the following system:
\begin{equation*}
A=
\begin{pmatrix}
.3 & .5 & .3\\
.6 & .6 & .2\\
.6 & .3 & .6
\end{pmatrix},
\ b=
\begin{pmatrix}
.6 \\
.3 \\
.2
\end{pmatrix}
\end{equation*}
and consider $\lambda=.6$.

Then the set $I_0=\{2,3\}$. From the entries of $A$ we can derive sets $C_1=\{2,3\},C_2=\{2\}$ and $C_3=\{3\}$.

For $W=\{1\}$, we have $C^W=\{C_1\}$. This set is a minimal covering of $I_0$, the minimal solution here is $z^W=z^{\{1\}}=(.6,0,0)$ and every $z\geq z^{\{1\}}$ is also a solution, i.e. $S^{\{1\}}=\{z^{\{1\}}\oplus v\colon v\in\B^{3}\}$.

Similarly, for $W=\{2,3\}$ we have the set $C^{W}=\{C_2,C_3\}$ which is minimal covering of $I_0$ with minimal solution $z^W= z^{\{2,3\}}=(0,.6,.6)$. Again, also vectors $z\geq z^{\{2,3\}}$ are solutions to the system, i.e. $S^{\{2,3\}}=\{z^{\{2,3\}}\oplus v\colon v\in\B^{3}\}$

Final solution set for the system \eqref{e:Axb} is then represented by the union of particular solution sets, $\bigcup_W S^W$.
\end{example}

%%%%%%%%%%%%%%%%%%%%%%%%%%%
\section{Max-min eigenproblem}
\label{s:mainres}
%---------------------------------------------------------------------------------------------
Let us first consider some two-dimensional examples that show
how the solution to the max-min eigenproblem~\eqref{e:eig} splits into 
subsets of $(K,L)$-eigenvectors.

%That means two-dimensional max-min $(K,L)$-eigenspaces (where $K$ ranges over all
%subsets of $\{1,2\}$ and $L$ is the complement of $K$).
%
\begin{example}\normalfont
\label{ex:1}
Take
\begin{equation}
A=
\begin{pmatrix}
.7 & .3\\
.2 & .5
\end{pmatrix}
\end{equation}
and consider $\lambda=.5$. Then the solution of \eqref{e:eig} is equivalent to the system
\begin{equation}
\label{e:first}
\max \bigl( \min (.7,x_1), \min(.3,x_2)\bigr) =\min(.5,x_1),
\end{equation}
\begin{equation}
\label{e:second}
\max \bigl( \min (.2,x_1), \min(.5,x_2)\bigr) =\min(.5,x_2)
\end{equation}
\if{
These two equations can be also written as
\begin{equation}
\label{e:first}
(.7 \wedge x_1) \vee (.3 \wedge x_2) = .5 \wedge x_1
\end{equation}
\begin{equation}
\label{e:second}
(.2 \wedge x_1) \vee (.5 \wedge x_2) = .5 \wedge x_2
\end{equation}
}\fi
\if{
Particular solution, the solution set for the equation \eqref{e:first} can be written as
\begin{equation}
\begin{split}\label{e:first-sol}
X^{\prime}=\Bigl\{ (x_1,x_2)\colon
               &(x_1=0.5 ) \ \vee
               (x_1 < 0.5 \ \wedge\ (0.3 \wedge x_2=x_1 ) \ \vee\ \\
          \vee\ &(x_1 < 0.5 \ \wedge\ x_1\leq 0.7 \ \wedge\ (0.3 \wedge x_2\leq x_1))
 \Bigr\}.
\end{split}
\end{equation}
Second particular solution, the solution set of \eqref{e:second} is
\begin{equation}
\begin{split}\label{e:second-sol}
X^{\prime\prime}=\Bigl\{ (x_1,x_2)\colon
               &(x_2\geq 0.5 ) \ \vee
               (x_2 < 0.5 \ \wedge\ (0.2 \wedge x_1=x_2 ) \ \vee\ \\
         \vee\ &(x_2 < 0.5 \ \wedge\ x_2\leq 0.5 \ \wedge\ (0.2 \wedge x_1\leq x_2))
\Bigr\}.
\end{split}
\end{equation}
}\fi
The solution set for~\eqref{e:first} is
\begin{equation}
\label{e:first-sol}
X^{\prime}=\Bigl\{(x_1,x_2)\colon
((x_1\geq .3)\vee (x_1\geq x_2))\wedge (x_1\leq .5)
\Bigr\},
\end{equation}
and the solution set for~\ref{e:second} is
\begin{equation}
\label{e:second-sol}
X^{\prime\prime}=\Bigl\{(x_1,x_2)\colon
(x_2\geq .2)\vee (x_1\leq x_2)\}
\Bigr\}.
\end{equation}
Then, the solution set to the eigenproblem is $X^{\prime} \bigcap X^{\prime\prime}$. Sets $X^{\prime}$, $X^{\prime\prime}$ and their intersection are displayed in Figure \ref{f:ex1}.

\begin{figure}
 \begin{center}
\hspace{-10pt}
  \includegraphics[width=\textwidth]{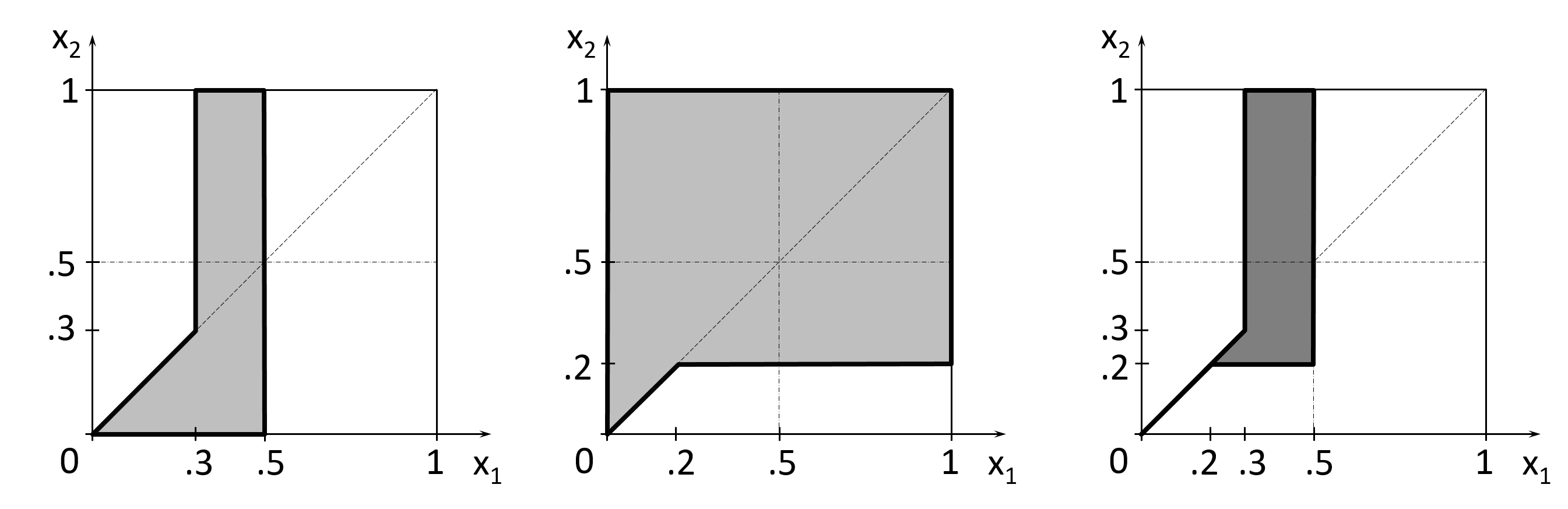}
  \caption{Sets $X^{\prime}$, $X^{\prime\prime}$ and their intersection: the $0.5$-eigenspace of $A$}\label{f:ex1}
  \end{center}
\end{figure}

In Figure \ref{f:exKL} we observe the effect of the value $\lambda$ on the final solution set.
The eigenvectors can be thus studied in individual areas (subsets) defined by $\lambda$.
The boundaries of these areas represented by the dashed line divide the solution set of our two-dimensional example into four areas (in figure quadrants $Q_1-Q_4$).

For the eigenvectors in $Q_1$ it holds that all $x_i\geq\lambda$ and thus we say that all $i\in L$ and $K=\emptyset$. We call these eigenvectors the background eigenvectors of $A$.
For $Q_2$ and $Q_3$ it holds that $x_i\leq\lambda$ for some $i \in K$ and some $x_i\geq\lambda$ for some $i \in L$.
In $Q_4$ all $x_i\leq \lambda$, it means that all $i \in K$ and $L=\emptyset$. We call these vectors the principal eigenvectors of $A$.

Note that in example \ref{ex:1} we have some ``genuine'' $(K,L)$-eigenvectors in the interior of $Q_2$,
which are neither principal nor background eigenvectors.

\begin{figure}
 \begin{center}
\hspace{-10pt}
  \includegraphics[height=120pt]{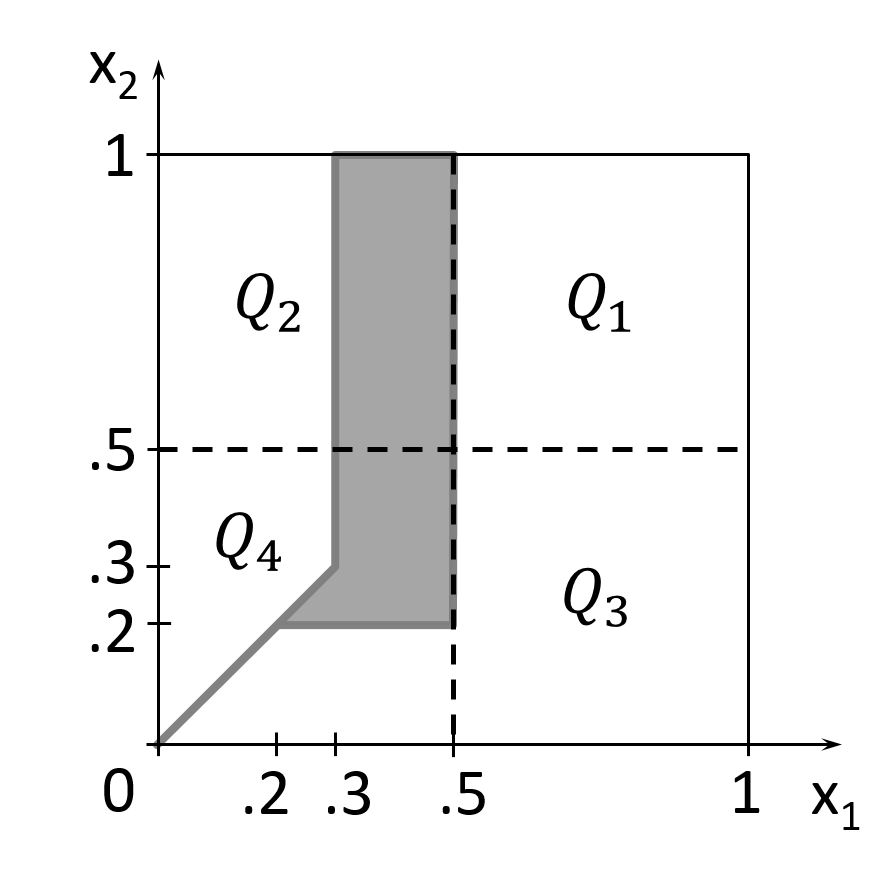}
  \caption{$(K,L)$-eigenvectors}\label{f:exKL}
  \end{center}
\end{figure}
\end{example}
\begin{example}\normalfont
\label{ex:2}
In this example, we take
\begin{equation}
B=
\begin{pmatrix}
.4 & .5\\
.2 & .5
\end{pmatrix}
\end{equation}
and consider the same $\lambda=.5$.
To solve \eqref{e:eig} for $A=B$ we need to solve the system
\begin{equation}
\label{e:firstB}
\max \bigl( \min (.4,x_1), \min(.5,x_2)\bigr) =\min(.5,x_1),
\end{equation}
\begin{equation}
\label{e:secondB}
\max \bigl( \min (.2,x_1), \min(.5,x_2)\bigr) =\min(.5,x_2)
\end{equation}
Note that~\eqref{e:secondB} is the same as~\eqref{e:second}.
The solution set for~\eqref{e:firstB} is
\begin{equation}
\label{e:first-solB}
Y^{\prime}=\Bigl\{(x_1,x_2)\colon
((x_1=x_2\leq .5)\vee (x_2\leq x_1\leq .4) \vee ((x_1\geq .5)\wedge (x_2\geq .5))
\Bigr\},
\end{equation}
and the solution set for~\eqref{e:secondB} is $Y^{\prime\prime}=X^{\prime\prime}$ expressed
in~\eqref{e:second-sol}.
Solution sets $Y^{\prime}$ and $Y^{\prime\prime}$ and their intersection
are depicted in Figure \ref{f:ex2} (there are no ``genuine'' $(K,L)$-eigenvectors in this example). 
\begin{figure}
 \begin{center}
\hspace{-10pt}
  \includegraphics[width=\textwidth]{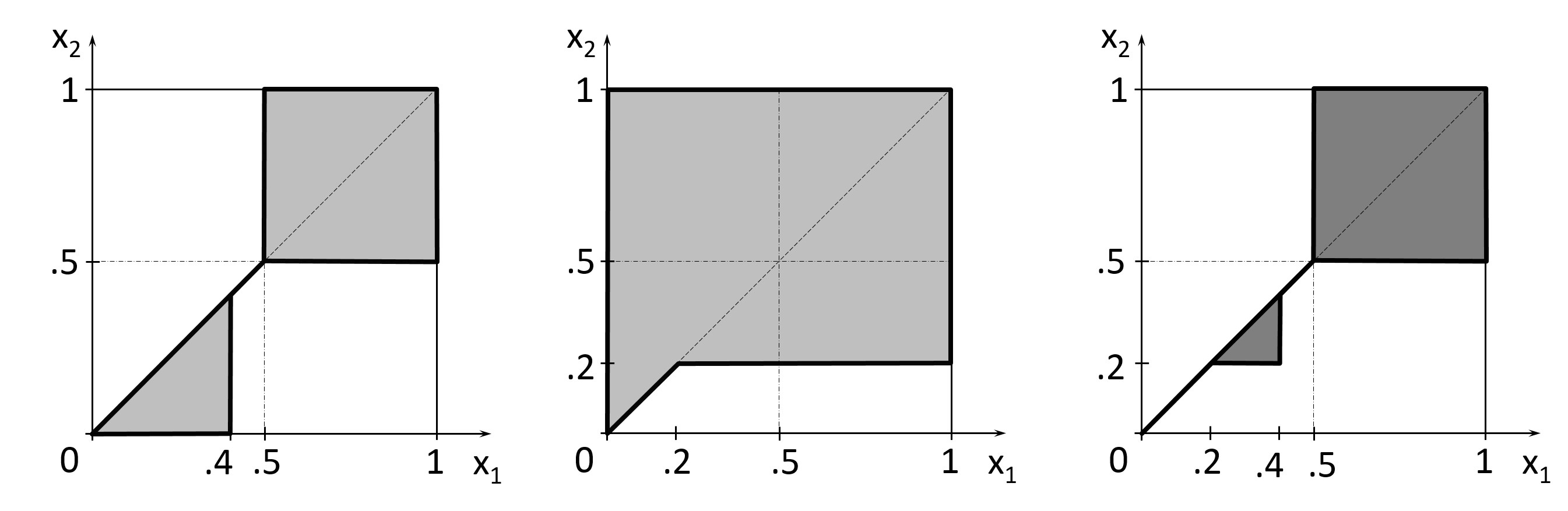}
  \caption{Sets $Y^{\prime}$ and $Y^{\prime\prime}$ and their intersection:
	the $0.5$-eigenspace of $B$}\label{f:ex2}
  \end{center}
\end{figure}
\end{example}

We are now going to give a theoretical description of background
eigenvectors, principal eigenvectors and $(K,L)$-eigenvectors in max-min algebra.

\subsection{Background $\lambda$-eigenvectors}
%---------------------------------------------------------------------------------------------

These are the vectors that satisfy $Ax=\lambda\mathbf{1}$ and $x_i\geq \lambda$ for all $i$.
Let us introduce the following notation:
\begin{equation}
\label{e:Klambda}
\begin{split}
N^{>\lambda}=\{k\colon \max\limits_{1\leq i\leq n} a_{ik}>\lambda\},\quad N^{\leq\lambda}=\{k\colon \max\limits_{1\leq i\leq n} a_{ik}\leq\lambda\}.%\\
%N^{=\lambda}&=\{k\colon \max\limits_{1\leq i\leq n} a_{ik}=\lambda\}
\end{split}
\end{equation}
The set of background eigenvectors can be described as follows.
\begin{proposition}
\label{p:backgr}
Let $A\in\B^{n\times n}$ and $\lambda\in\B$.
Then the set of background $\lambda$-eigenvectors of $A$ is nonempty
if and only if
\begin{equation}
\label{e:backgr-nonempty}
\max\limits_{1\leq j\leq n} a_{ij}\geq \lambda\quad\forall i.
\end{equation}
If~\eqref{e:backgr-nonempty} holds then the set of background eigenvectors is
given by
\begin{equation}
\label{e:backgr}
\{x\colon x_k=\lambda\ \text{for $k\in N^{>\lambda}$}, \lambda\leq x_k\leq 1\ \text{for
$k\in N^{\leq\lambda}$}\}.
\end{equation}
\end{proposition}
\begin{proof}
Observe first that if there exist $i$ with $a_{ij}<\lambda$ for all $j$, then also
$a_{ij}x_j<\lambda$ for all $j$ implying that $\bigoplus_{j=1}^n a_{ij}x_j=\lambda$ cannot hold and
the set of background eigenvectors is empty. If~\eqref{e:backgr-nonempty} holds then the constant
vector $x=\lambda\mathbf{1}$ satisfies $Ax=\lambda\mathbf{1}$ hence the set of background eigenvectors is nonempty.

If $k\in N^{>\lambda}$ then there exists $i$ that $a_{ik}>\lambda$, and we need $x_k=\lambda$ to make
sure that $\bigoplus_{j=1}^n a_{ij}x_j\leq\lambda$. This shows that the set of background $\lambda$-eigenvectors is a subset of~\eqref{e:backgr}.

Now take a vector from~\eqref{e:backgr}. Obviously, it satisfies $x_i\geq\lambda$ for each $i$.
Since also $\max_j a_{ij}\geq\lambda$ for all $j$, we have $\bigoplus_j a_{ij}x_j\geq \lambda$. But we
also have $\bigoplus_j a_{ij}x_j\leq \lambda$. Indeed, since $x_j=\lambda$ in~\eqref{e:backgr}
whenever $j\in N^{>\lambda}$, that is, whenever  there exists $i$ with $a_{ij}>\lambda$, we have
$a_{ij}x_j\leq\lambda$ for all such $j$ and all $i$. For $j\in N^{\leq\lambda}$, inequality
$a_{ij}x_j\leq\lambda$ follows from $\max_i a_{ij}\leq \lambda$ (by the definition of $N^{\leq\lambda}$).

The proof is complete.
\end{proof}

Note that the set of background eigenvectors is a max-min convex set, and not a max-min space. The proof of the following corollary of Proposition~\ref{p:lambdaboxgen} is straightforward and will be omitted.

\begin{corollary}
\label{c:lambdaboxgen}
If the set of background eigenvectors is non-empty, then it is the max-min convex hull of $e^{\lambda 0}=\lambda\mathbf{1}$ and vectors 
$e^{\lambda k}$ for $k\in N^{\leq\lambda},$ for which 
$$
e^{\lambda k}_i=
\begin{cases}
1, &\text{if $k=i$,}\\
\lambda, &\text{otherwise.}
\end{cases}
$$
\end{corollary}

The following proposition will be helpful when describing more general sets of the form~\eqref{e:backgr}
as max-min convex hulls.
\begin{proposition}
\label{p:lambdaboxgen}
Suppose that
\begin{equation}
\label{e:S}
S=\{x\colon x_k=\lambda\ \text{for $k\in N_1$},\  \lambda\leq x_k\leq 1\ \text{for
$k\in N_2$}\},
\end{equation}
where $N_1$ and $N_2$ are such that $N_1\cup N_2=N(=\{1,\ldots,n\})$ and $N_1\cap N_2=\emptyset$.
Then
\begin{equation}
\label{e:N_2}
S=\{\lambda \mathbf{1}\oplus z_{N_2}\colon z_{N_2}\in\B^{|N_2|}\}
\end{equation}
%where $\Lambda$ is defined by
%\begin{equation}
%\label{e:Lambda}
%\begin{split}
%& (\Lambda)_{ij}=
%\begin{cases}
%1, &\text{if $i=j$},\\
%\lambda, &\text{if $i\neq j$},
%\end{cases}\\
%&i,j\in N.
%\end{split}
%\end{equation}
\end{proposition}

\subsection{Principal $\lambda$-eigenvectors}
%---------------------------------------------------------------------------------------------
%
These are the vectors that satisfy $Ax=x$ and $x_i\leq\lambda$ for all $i$.
Description of a generating set of the space of principal
max-min $\lambda$-eigenvectors is given below. Observe that any
principal $\lambda$-eigenvector is a principal eigenvector and therefore we can apply
Theorem~\ref{t:principal}.
\begin{corollary}[Gondran-Minoux~\cite{GM:08}, Ch. 6, Corollary 3.5]
\label{c:principal}
The set of principal $\lambda$-eigenvectors of $A\in\B^{n\times n}$ is the max-min  column space
$\{A^*_{\lambda} y\colon y\in\B^n\}$, where the columns of
$A^*_{\lambda}$ are defined by
\begin{equation}
\label{e:genset-principal}
(A^*_{\lambda})_{\cdot i}= \lambda a_{ii}^+ (A^*)_{\cdot i}\colon i=1,\ldots,n.
\end{equation}
More precisely, each vector of~\eqref{e:genset-principal} is a principal $\lambda$-eigenvector,
and each principal $\lambda$-eigenvector $x$ can be represented as
\begin{equation}
\label{e:repr-principal}
x=\bigoplus_{i\in C(A,x)} \lambda x_ia_{ii}^+ (A^*)_{\cdot i},
\end{equation}
where $C(A,x)\subseteq N$ is a set containing a node from each strongly connected component of $\Sat(A,x)$.
\end{corollary}
\begin{proof}
The claim follows as an easy corollary of Theorem~\ref{t:principal}. Indeed, since
 $a_{ii}^+(A^*)_{\cdot i}$ satisfies $Ax=x$, so does
$\lambda a_{ii}^+ (A^*)_{\cdot i}$. As components of this vector do not exceed $\lambda$,
it is a principal $\lambda$-eigenvector. Letting $x$ be a principal $\lambda$-eigenvector, we see that it satisfies~\eqref{e:repr-princ} since it satisfies $Ax=x$. Equation~\eqref{e:repr-principal}
follows from~\eqref{e:repr-princ} after
multiplying both parts of~\eqref{e:repr-princ} by $\lambda$ and observing that $\lambda x=x$ since $x_i\leq \lambda$ for all $i$.
\end{proof}

\if{
\begin{corollary}
\label{c:principal-conv}
The set of max-min background $\lambda$-eigenvectors of $A$ is the max-min convex hull of
$\{p^i\colon i=0,1,\ldots,n\}$ where
\begin{equation}
\label{e:pidef}
p^0=\mathbf{0}, \quad p^i=\lambda a_{ii}^+ {(A^*)_{\cdot i},\quad i=1,\ldots,n.
\end{equation}
\end{corollary}
}\fi

\subsection{$(K,L)$ $\lambda$-eigenvectors}
%---------------------------------------------------------------------------------------------

Now we consider $(K,L)$ max-min $\lambda$-eigenvectors, i.e., $x\in\B^n$ such that
$Ax=\lambda x$,
$x_i\leq \lambda$ for $i\in K$ and
$x_i\geq \lambda$ for $i\in L$, where $K,L\subseteq\{1,\ldots,n\}$ are such that
$K\cup L=\{1,\ldots,n\}$ and $K\cap L=\emptyset$.

By this definition, every $(K,L)$ $\lambda$-eigenvector is of the form $x=(x_L \  x_K)$  (in a suitable permutation of indices ) with
%for fixed $(K,L)$-partition
%
\begin{align}
\label{Keqn}
A_{KK}x_K\oplus A_{KL}x_L &=  x_K,\quad x_K\leq \lambda \mathbf{1}_K\\
\label{Leqn}
A_{LK}x_K\oplus A_{LL}x_L &= \lambda\mathbf{1}_L, \quad x_L\geq \lambda \mathbf{1}_L
\end{align}

We start with describing the solvability and the set of solutions for
\eqref{Keqn}.
\begin{proposition}\label{solvability_1}
Equation \eqref{Keqn} is solvable (together with $x_K\leq \lambda \mathbf{1}_K$) with respect to 
$x_K$ if and only if
\begin{equation}
\label{Keqn-cond}
(A_{KK})^*A_{KL}x_L\leq \lambda\mathbf{1}_K,
\end{equation}
and then the set of its solutions for fixed $x_L$ is given by
\begin{equation}
\label{e:SK}
\begin{split}
S_K(x_L,\lambda)&=\{(A_{KK})^*A_{KL}x_L\oplus v\colon A_{KK}v=v,\  v\leq\lambda\mathbf{1}_K\}\\
&=\{(A_{KK})^*A_{KL}x_L\oplus (A_{KK})^*_{\lambda}z_K\colon z_K\in\B^{|K|}\}
\end{split}
\end{equation}
where $(A_{KK})^*_{\lambda}$ is defined as in~\eqref{e:genset-principal}.
\end{proposition}
\begin{proof}
the assertion follows from Theorem~\ref{t:schneider}.
\end{proof}

Before describing the solvability of \eqref{Leqn}, denote
\begin{equation}
\label{e:L1L2}
L_1=\{i\in L\colon \bigoplus_{j\in L}a_{ij}\geq \lambda\},\quad
L_2=\{i\in L\colon \bigoplus_{j\in L}a_{ij}< \lambda\}
\end{equation}
\begin{proposition}\label{solvability_2}
Equation \eqref{Leqn}, considered together with $x_K\leq \lambda \mathbf{1}_K$, is equivalent to
\begin{align}
\label{L1eqn}
A_{L_1L}x_L &=  \lambda\mathbf{1}_{L_1},\quad x_L\geq \lambda \mathbf{1}_L,\\
\label{L2eqn}
A_{L_2K}x_K &= \lambda\mathbf{1}_{L_2}, \quad x_K\leq \lambda \mathbf{1}_K.
\end{align}
\end{proposition}
\begin{proof}
	In terms of the notation \eqref{e:L1L2}, equation \eqref{Leqn}
 is written as follows:
\begin{eqnarray}
\label{L1eqnini}
A_{L_1K}x_K\oplus A_{L_1L}x_L =  \lambda\mathbf{1}_{L_1},\\
\label{L2eqnini}
A_{L_2K}x_K\oplus A_{L_2L}x_L = \lambda\mathbf{1}_{L_2},\\
\label{KLineqs}
x_L\geq \lambda \mathbf{1}_L,\quad x_K\leq \lambda \mathbf{1}_K.
\end{eqnarray}
We now show that \eqref{L1eqnini}, \eqref{L2eqnini}, \eqref{KLineqs} are equivalent to
\eqref{L1eqn} and \eqref{L2eqn}. Indeed, $x_L\geq \lambda \mathbf{1}_L$ together with $\bigoplus_{j\in L}a_{ij}\geq \lambda$
for $i\in L_1$ imply that $A_{L_1L}x_L \geq  \lambda\mathbf{1}_{L_1}$ holds for any feasible $x$,
and this makes the first term in~\eqref{L1eqnini} redundant, also since $x_K\leq \lambda \mathbf{1}_K$
implies $A_{L_1K}x_K\leq \lambda\mathbf{1}_{L_1}$. As for the second term in~\eqref{L2eqnini}, we have
$\bigoplus_{j\in L}a_{ij}< \lambda$ for $i\in L_2$ implying that $A_{L_2L}x_L<\lambda\mathbf{1}_{L_2}$
and making this term redundant as well.
\end{proof}

To solve~\eqref{L1eqn}
let us introduce the following notation in analogy with~\eqref{e:Klambda}:
\begin{align}
\label{e:blambda>}
L^{>\lambda,L_1} &=\{k\in L\colon \max\limits_{i\in L_1} a_{ik}>\lambda\}\\
\label{e:LbKlambda}
L^{>\lambda,K}  &=\{\ell\in L\colon \max_{k\in K}((A_{KK})^*A_{KL})_{k\ell}>\lambda\}\\
L^{\prime} &= L^{>\lambda,L_1}\cup L^{>\lambda,K},\quad \Tilde{L}=L\backslash L^{\prime}.
\end{align}

\begin{proposition}\label{solvability_3}
The solution set to \eqref{L1eqn} with condition~\eqref{Keqn-cond} is given by
\begin{equation}
\label{e:SKL-rep-L}
\begin{split}
& S_L(\lambda)=
\{\lambda\mathbf{1}_L \oplus z_{\Tilde{L}}\colon z_{\Tilde{L}}\in\B^{|\Tilde{L}|}\}.
\end{split}
\end{equation}

\end{proposition}
\begin{proof}
Following Proposition~\ref{p:backgr}, the set of solutions to~\eqref{L1eqn} is
\begin{equation}
\label{e:xLsol}
\{x_L\colon x_{\ell}=\lambda\ \text{for $\ell\in L^{>\lambda,L_1}$}, \lambda\leq x_{\ell}\leq 1\ \text{for
$\ell\notin L^{>\lambda,L_1}$}\}.
\end{equation}
However, we also have~\eqref{Keqn-cond}, which is satisfied whenever
\begin{equation}
\label{e:xLsol1}
x_{\ell}=\lambda \text { for all } \ell \in L^{>\lambda,K}.
\end{equation}
	The set of solutions to~\eqref{L1eqn} with respect to both conditions \eqref{e:xLsol} and \eqref{e:xLsol1} can be written as
\begin{equation}
\label{e:SL}
S_L(\lambda) = \{x_L\colon x_{\ell}=\lambda\ \text{for $\ell\in L^{\prime}$},
%L^{>\lambda,L_1}\cup L^{>\lambda,K}$,
\lambda\leq x_{\ell}\leq 1\ \text{for
$\ell\in\Tilde{L}$}\}.
\end{equation}
Using Proposition~\ref{p:lambdaboxgen} we can express $S_L(\lambda)$ as shown in
\eqref{e:SKL-rep-L}.
\end{proof}

\begin{proposition}\label{solvability_4}
The set of vectors $x_K$ solving \eqref{Keqn} for a given $x_L=\lambda\mathbf{1}_L \oplus z_{\Tilde{L}}$ is of the form
\begin{equation}
\label{e:SKL-rep-kk}
\begin{split}
& S_K(x_L,\lambda)=\{\lambda(A_{KK})^*A_{KL}\mathbf{1}_L\oplus (A_{KK})^*A_{K\Tilde{L}}z_{\Tilde{L}}
\oplus (A_{KK})^*_{\lambda}z_K\colon z_K\in\B^{|K|}\}.
\end{split}
\end{equation}
\end{proposition}

\begin{proof}
The computation of $S_K(x_L,\lambda)$ in view of \eqref{e:SK} uses $x_L\in S_L(\lambda)$ according to  \eqref{e:SKL-rep-L} and distributivity.
\begin{equation}
\label{e:SKL-rep-k}
\begin{split}
& S_K(x_L,\lambda)=\{(A_{KK})^*A_{KL}x_L\oplus (A_{KK})^*_{\lambda}z_K\colon z_K\in\B^{|K|}\}\\
%&= \{\lambda(A_{KK})^*A_{KL}\mathbf{1}_L\oplus (A_{KK})^*A_{KL}\Lambda_{L\Tilde{L}}z_{\Tilde{L}}
%\oplus (A_{KK})^*_{\lambda}z_K\colon z_K\in\B^{|K|}\}\\
&= \{\lambda(A_{KK})^*A_{KL}\mathbf{1}_L\oplus (A_{KK})^*A_{K\Tilde{L}}z_{\Tilde{L}}
\oplus (A_{KK})^*_{\lambda}z_K\colon z_K\in\B^{|K|}\}.
\end{split}
\end{equation}
\if{
To deduce the last expression we observed that
\begin{equation*}
\begin{split}
&(A_{KK})^*A_{KL}\Lambda_{L\Tilde{L}}z_{\Tilde{L}}=
(A_{KK})^*A_{K\Tilde{L}}\Lambda_{\Tilde{L}\Tilde{L}}z_{\Tilde{L}}\oplus
(A_{KK})^*A_{KL'}\Lambda_{L'\Tilde{L}}z_{\Tilde{L}}\\
&\leq (A_{KK})^*A_{K\Tilde{L}}\Lambda_{\Tilde{L}\Tilde{L}}z_{\Tilde{L}}\oplus
\lambda (A_{KK})^*A_{KL} \mathbf{1}_{L},
\end{split}
\end{equation*}
using the definition~\eqref{e:Lambda}.
}\fi
\end{proof}

%\begin{proposition}\label{solvability_5}
%
%\end{proposition}
%
%\if{
%}\fi

Substituting $x_K\in S_K(x_L,\lambda)$ given by~\eqref{e:SKL-rep-kk} into~\eqref{L2eqn}, we obtain 
the following equation:
\begin{equation}
\label{e:finaleq}
\lambda A_{L_2K}(A_{KK})^*A_{KL}\mathbf{1}_L\oplus
A_{L_2K}(A_{KK})^*A_{K\Tilde{L}}z_{\Tilde{L}}
\oplus A_{L_2K}(A_{KK})^*_{\lambda}z_K=\lambda\mathbf{1}_{L_2}
\end{equation}

It can be seen that the system \eqref{e:finaleq} is of the form
$A'z'\oplus b'=\lambda \mathbf{1}$ where
\begin{equation}
\label{e:Abz}
\begin{split}
A'&=(A_{L_2K}(A_{KK})^*A_{K\Tilde{L}},\   A_{L_2K}(A_{KK})^*_{\lambda}),\\
b'&= \lambda A_{L_2K}(A_{KK})^*A_{KL}\mathbf{1}_L,\  z'=(z_{\Tilde{L}}\ z_K)^T.
\end{split}
\end{equation}
This observation enables us to describe the solution $z'=(z_{\Tilde{L}}\ z_K)^T$ using the 
results of Subsection~\ref{s:special}, as below. 

\begin{proposition}\label{solvability_7}
The minimal solutions $(z')^W=(z_K^W z_{\Tilde{L}}^W)$ of $A'z'\oplus b'=\lambda \mathbf{1}$ with 
$A'$ and $b'$ as in \eqref{e:Abz} are defined by 
\begin{equation}
(z')^W_j=
\begin{cases}
\lambda, &\text{if $j\in W$},\\
0, &\text{otherwise}.
\end{cases}
\end{equation}
where $\{\cup C_j\colon j\in W\}$ with 
\begin{equation}
\label{e:Cjnew}
C_j=\{i\in I_0\colon (A')_{ij}=\lambda\}\quad  \text{for $j\in\Tilde{N}$},
\end{equation}
is a minimal covering of
\begin{equation}
\label{e:I0new}
I_0=\{i\in L_2\colon A_{iK}(A_{KK})^*A_{KL}\mathbf{1}_L<\lambda\}.
\end{equation}
\end{proposition}

\begin{proposition}\label{allsolutions}
The set of all solutions of $A'z'\oplus b'=\lambda \mathbf{1}$ with 
$A'$ and $b'$ as in \eqref{e:Abz} is given by $\cup_W S_W$ where the union is taken over all 
minimal coverings $\{\cup C_j\colon j\in W\}$ of $I_0$ and 
$S_W=\{z^W\oplus v_{\Tilde{N}}\colon v_{\Tilde{N}}\in\B^{|\Tilde{N}|}\}$
%with $\Lambda^W\in\B^{|\Tilde{N}|\times |\Tilde{N}|}$ defined by 
%\begin{equation*}
%(\Lambda^W)_{ij}
%=
%\begin{cases}
%1, & \text{if $i=j$},\\
%z_i^W, & \text{if $i\neq j$}.
%\end{cases}
%\end{equation*}
\end{proposition}

\begin{proof}[Proof of Proposition~\ref{solvability_7} and~\ref{allsolutions}]
Observe that all entries of $A_{KK}^*A_{K\Tilde{L}}$ do not exceed $\lambda$ by~\eqref{e:LbKlambda}
and the definition of $\Tilde{L}$, and that all coefficients of $(A_{KK})^*_{\lambda}$ do not
exceed $\lambda$ by~\eqref{e:genset-principal}.
Hence the entries of $A'$ and $b'$ do not exceed $\lambda$, and all solutions of
$A'z'\oplus b'=\lambda\mathbf{1}$ can be found as in Corollary~\ref{c:Axb} and 
using~\eqref{e:SWM}, with $A'$ and $b'$
instead of $A$ and $b$,
$\Tilde{N}=\Tilde{L}\cup K$ instead of $N$.
\end{proof}

%Using Proposition~\ref{p:xboxgen} to express the solution sets $S^W:=\{z'\colon (z')^W\leq z'\leq\mathbf{1}\}$
%algebraically, we can substitute the result back in~\eqref{e:SKL-rep-L} and \eqref{e:SKL-rep-kk} thus obtaining
%the following description of $(K,L)$-eigenvectors.

The next theorem, which describes the set of $(K,L)$ $\lambda$-eigenvectors, is the main result of the section.
This theorem follows from the arguments written above, but we also include a formal proof based on backtracking the above arguments.

\begin{theorem}
\label{t:mainres}
Let $A\in\B^{n\times n}$, $\lambda\in\B$ and $K,L$ such that $K\cup L=N$ and $K\cap L=\emptyset$
be given. Then $(K,L)$-eigenvector exists if and only if~\eqref{e:finaleq} is solvable, which happens if
and only if $I_0=\cup_{j\in\Tilde{N}} C_j$, with $I_0$ and $C_j$ defined as in~\eqref{e:Cjnew}
and~\eqref{e:I0new} and $\Tilde{N}=\Tilde{L}\cup K$.

In this case, $x=(x_K\ x_L)^T$ is a $(K,L)$-eigenvector of $A$ with eigenvalue $\lambda$
if and only if $x_K$ and $x_L$ can be expressed as follows:

\begin{equation}
\label{e:SKL-rep-final}
\begin{split}
x_L&= \lambda\mathbf{1}_L \oplus  v_{\Tilde{L}},\\
x_K&= \lambda(A_{KK})^*A_{KL}\mathbf{1}_L\oplus (A_{KK})^*A_{K\Tilde{L}} v_{\Tilde{L}}
\oplus (A_{KK})^*_{\lambda}(z_{K}^W \oplus v_{K}),
\end{split}
\end{equation}
where $v_{K}\in\B^{|K|},$
$v_{\Tilde{L}}\in\B^{|\Tilde{L}|},$ and $z^W$ is the minimal solution
of~\eqref{e:finaleq} corresponding to a minimal covering $\{C_j\colon j\in W\}$ of $I_0$.
%and $\Tilde{L}=L\backslash(L^{>\lambda,L_1}\cup L^{>\lambda,K)$ (with $L^{>\lambda,L_1}$ and $L^{>\lambda, K}$ defined as in~\eqref{e:blambda>} and~\eqref{e:LbKlambda}).
\end{theorem}
\begin{proof}
Consider a set of vectors $z_{\Tilde{N}}=(z_K\ z_{\Tilde{L}})$ described by 
\begin{equation*}
%\begin{split}
z_K=z_K^W\oplus v_K, \quad 
z_{\Tilde{L}}=z_{\Tilde{L}}^W\oplus v_{\Tilde{L}},\quad v_{K}\in\B^{|K|},\quad 
v_{\Tilde{L}}\in\B^{|\Tilde{L}|}
%\end{split}
\end{equation*}
for a minimal $W$ such that $\cup_{j\in W} C_j=I_0$. By Proposition~\ref{solvability_7} 
and Proposition~\ref{allsolutions}, the union of these sets over $W$ comprises the set of all solutions 
to~\eqref{e:finaleq}, which is solvable if and only if $\cup_{j\in\Tilde{N}} C_j=I_0$. 

Equation~\eqref{e:finaleq}, if it is solvable, is the same as $A_{L_2K}x_K=\lambda \mathbf{1}_{L_2}$ of~\eqref{L2eqn}, where $x_K\in S_K(x_L,\lambda)$ with the latter set described in~\eqref{e:SKL-rep-kk}. 
Note that the condition $x_K\leq \lambda \mathbf{1}_K$
immediately follows from~\eqref{e:SKL-rep-kk}. Then we see that~\eqref{e:SKL-rep-final} describes all vectors 
$(x_K\ x_L)$ such that $x_L\in S_L(\lambda)$, $x_K\in S_K(x_L,\lambda)$ and condition~\eqref{L2eqn} is 
satisfied, after observing that $x_L=\lambda\mathbf{1}_L\oplus z^W_{\Tilde{L}}\oplus v_{\Tilde{L}}=
\lambda\mathbf{1}\oplus v_{\Tilde{L}}$ (since each component of $z^W$ is less than or equal to $\lambda$). 

By Proposition~\ref{solvability_4}, $S_K(x_L,\lambda)$ is the set of all vectors $x_K$ solving~\eqref{Keqn}
for a given $x_L=\lambda \mathbf{1}_L\oplus z_{\Tilde{L}}$, which is by 
Proposition~\ref{solvability_3} a general solution to~\eqref{L1eqn} with condition~\eqref{Keqn-cond}.
Observe that by Proposition~\ref{solvability_1} condition~\eqref{Keqn-cond} is equivalent to solvability 
of~\eqref{Keqn} with respect to $x_K$ for a given $x_L$.  Thus~\eqref{e:SKL-rep-final} describes all vectors 
$(x_K\ x_L)$ that satisfy~\eqref{Keqn},~\eqref{L1eqn} and~\eqref{L2eqn} simultaneously. 

Furthermore, in view of Proposition~\ref{solvability_2},   \eqref{L1eqn}
and~\eqref{L2eqn} can be replaced by~\eqref{Leqn}, implying that 
~\eqref{e:SKL-rep-final} yields all solutions to~\eqref{Keqn} and~\eqref{Leqn} if~\eqref{e:finaleq} is 
solvable. 

If~\eqref{e:finaleq} is not solvable then~\eqref{L2eqn} is not compatible with~\eqref{Keqn} 
and~\eqref{L1eqn}. The proof is complete. 
\end{proof}

Let us also note the following special case of the above considerations.

\begin{corollary}
\label{solvability_5}
If $L_2=\emptyset$ then the solution set to \eqref{Keqn}, \eqref{Leqn} is the set of 
$x=(x_L\ x_K)$ given by
\begin{equation}
\label{e:SLSK}
\begin{split}
x_L&= \lambda\mathbf{1}_L \oplus z_{\Tilde{L}},\\
x_K&= \lambda(A_{KK})^*A_{KL}\mathbf{1}_L\oplus (A_{KK})^*A_{K\Tilde{L}}
z_{\Tilde{L}}\oplus (A_{KK})^*_{\lambda}z_{K},
\end{split}
\end{equation}
where  parameters $z_{\Tilde{L}}\in\B^{|\Tilde{L}|},\ z_K\in\B^{|K|}$ have arbitrary values.
\end{corollary}

In this case~\eqref{e:finaleq} is trivially solvable, as an ``empty equation''.

Theorem~\ref{t:mainres} gives us a clear way to generate all $(K,L)$-eigenvectors. If we want to ''test'' its validity, we may recall that, since it can be easily shown that $Ax\leq x$ and $A^*x=x$ are equivalent~\cite{GM:08} and since any $\lambda$-eigenvector satisfies $Ax=\lambda x\leq x$, we should have $A^*x=x$ for any $\lambda$-eigenvector $x$. 
Let us show that
any vector given by~\eqref{e:SKL-rep-final} 
satisfies this property.

\begin{corollary}
Let $x=(x_K\ x_L)$ satisfy~\eqref{e:SKL-rep-final}. Then we have $A^*x=x$.
\end{corollary}
\begin{proof}
Since $A^*x=x$ is equivalent to $Ax\leq x$, we will prove the latter. 
Inequality $Ax\leq x$ can be written as the following four:
\begin{equation*}
A_{KK}x_K\leq x_K,\  A_{KL}x_L\leq x_K,\  A_{LK}x_K\leq x_L,\  A_{LL}x_L\leq x_L.
\end{equation*}

The inequality $A_{KK}x_K\leq x_K$ follows since $x_K$, by the second equation of~\eqref{e:SKL-rep-final},
satisfies $(A_{KK})^*v=x_K$ for some $v$ and we have $A_{KK}(A_{KK})^*=(A_{KK})^+\leq (A_{KK})^*$.
As for the inequality $A_{KL}x_L\leq x_K$, we have
\begin{equation*}
%\begin{split}
A_{KL}x_L=\lambda A_{KL}\mathbf{1}_L \oplus A_{K\Tilde{L}}v_{\Tilde{L}}\leq x_K.
%& =\lambda A_{KL}\mathbf{1}_L \oplus A_{KL}\Lambda_{L\Tilde{L}}(z_{\Tilde{L}}^W \oplus \Lambda^W_{\Tilde{L}\Tilde{N}} v_{\Tilde{N}})\\
%&= \lambda A_{KL}\mathbf{1}_L \oplus (A_{KL'}\Lambda_{L'\Tilde{L}}\oplus  A_{K\Tilde{L}}
%\Lambda_{\Tilde{L}\Tilde{L}}) (z_{\Tilde{L}}^W \oplus \Lambda^W_{\Tilde{L}\Tilde{N}} v_{\Tilde{N}})\\
%&= \lambda A_{KL}\mathbf{1}_L \oplus  A_{K\Tilde{L}}
%\Lambda_{\Tilde{L}\Tilde{L}}  (z_{\Tilde{L}}^W \oplus \Lambda^W_{\Tilde{L}\Tilde{N}} v_{\Tilde{N}})\leq x_K.
%\end{split}
\end{equation*}
As for the inequality $A_{LK}x_K\leq x_L$, it follows since $x_L$ and $x_K$ of ~\eqref{e:SKL-rep-final} 
satisfy $x_L\geq \lambda\textbf{1}$ and $x_K\leq\lambda\textbf{1}$.

It remains to prove the inequality $A_{LL}x_L\leq x_L$, which we are going to do for arbitrary 
$x_L$ belonging to the set defined in~\eqref{e:xLsol}, using that this set contains \eqref{e:SKL-rep-L}
and hence any $x_L$ given by the first equation of~\eqref{e:SKL-rep-final}. Denoting 
$L_3=L^{>\lambda,L_1}$ and $L_4=L\backslash L^{>\lambda,L_1}$ 
and seeing that $x_{L_3}=\lambda\mathbf{1}_{L_3}$ and $x_{L_4}\geq\lambda\mathbf{1}_{L_4}$,
we have to prove
\begin{equation*}
\lambda A_{L_3L_3}\textbf{1}_{L_3}\leq \lambda\textbf{1}_{L_3},\  
A_{L_3L_4}x_{L_4}\leq \lambda\textbf{1}_{L_3},\  \lambda A_{L_4L_3} \textbf{1}_{L_3}\leq x_{L_4},\  
A_{L_4L_4}x_{L_4}\leq x_{L_4}.
\end{equation*}
 The first and the third of these inequalities are obvious. As for the inequalities 
$A_{L_3L_4}x_{L_4}\leq \lambda\textbf{1}_{L_3}$ and $A_{L_4L_4}x_{L_4}\leq x_{L_4}$, we observe using the 
definition of $L^{>\lambda,L_1}$~\eqref{e:blambda>} and $L_2$ in~\eqref{e:L1L2} that all entries in 
$A_{LL_4}$ do not exceed $\lambda$, and recall that $x_{L_4}\geq\lambda\textbf{1}_{L_4}$. 
Thus the inequality $A_{LL}x_L\leq x_L$ also follows and the proof is complete.  
\end{proof}

\if{
\begin{proof}
Clearly, every $(K,L)$ $\lambda$-eigenvector is of the form $x=(x_L \  x_K)$ fulfilling~\eqref{Keqn} and~\eqref{Leqn} and by Proposition~\ref{solvability_2}, equation~\eqref{Leqn} is equivalent to~\eqref{L1eqn} (which contains only $x_L$) and~\eqref{L2eqn} (which contains only $x_K$).

Viewing~\eqref{Keqn} as a system with unknown vector $x_K$, we express its solvability 
for a given $x_L$ by condition~\eqref{Keqn-cond}. Therefore, as~\eqref{e:SKL-rep-L} expresses the solution 
set to~\eqref{L1eqn} under~\eqref{Keqn-cond}, the set of all $x_L$ such that $(x_L,x_K)$ is a $(K,L)$
$\lambda$-eigenvector is given by all vectors in $S_L(\lambda)$ of~\eqref{e:SKL-rep-L} for which~\eqref{L2eqn}
is solvable (in view of the dependence of $x_K$ on $x_L$, which is discussed below).   

For a given $x_L\in S_L(\lambda)$ as in~\eqref{e:SKL-rep-L}, the set of all $x_K$ solving~\eqref{Keqn}
is described by $S_K(x_L,\lambda)$ of~\eqref{e:SK}, which is then recast as~\eqref{e:SKL-rep-kk}. 
We now use the description of $S_L(\lambda)$ of~\eqref{e:SKL-rep-L}, 
$S_K(x_L,\lambda)$ of~\eqref{e:SK} and condition~\eqref{L2eqn} to prove the claim. 
For this we first recast~\eqref{L2eqn} as~\eqref{e:finaleq}. The latter condition is of the form $A'z'\oplus b'=\lambda \mathbf{1}$ where $z'=(z_{\Tilde{L}}\ z_K)^T$ and $A'$ and $b'$ are as in~\eqref{e:Abz}. The full solution set of this equation is described by 
Proposition~\ref{solvability_7} and Proposition~\ref{allsolutions}. Substituting this 
solution into~\eqref{e:SKL-rep-L} and~\eqref{e:SK} we obtain the claim. 
\end{proof}
}\fi

Let us now describe the set of $(K,L)$-eigenvectors as a max-min convex hull of a finite number of points.
\begin{corollary}
\label{c:mainres}
If the set of $(K,L)$-eigenvectors associated with eigenvalue $\lambda$ is non-empty, then it is the max-min convex hull of the points $x^{W,0}=(x_K^{W,0}\; x_L^{W,0})$, $x^{W,i}=(x_K^{W,i}\; x_L^{W,i})$ for $i\in \Tilde{L}$ and 
$x^{W,k}=(x_K^{W,k}\; x_L^{W,k})$ for $k\in K$, where
\begin{equation}
\label{e:points}    
\begin{split}
x_L^{W,0}&=\lambda\mathbf{1}_L, \quad x_K^{W,0}=
(A_{KK})^*A_{KL}x_L^{W,0}\oplus (A_{KK})^*_{\lambda}z_K^W,\\ 
 x_L^{W,i}&=\lambda\mathbf{1}_L\oplus e_i,\quad 
 x_K^{W,i}=(A_{KK})^*A_{KL}x_L^{W,i}\oplus (A_{KK})^*_{\lambda}z_K^W,\\
 x_L^{W,k}&=\lambda\mathbf{1}_L,\quad x_K^{W,k}=(A_{KK})^*A_{KL}x_L^{W,k}\oplus (A_{KK})^*_{\lambda}(z_K^W\oplus e_k),
\end{split}    
\end{equation}
and $W$ ranges over all minimal coverings $\{C_j\colon j\in W\}$ of $I_0$.
\end{corollary}
\begin{proof}
Using that $\B^{|\Tilde{N}|}=\B^{|K|}\times \B^{|\Tilde{L}}|$ is a max-min convex set generated by the zero vector $\mathbf{0}$ and unit vectors $e_i$ for $i\in\Tilde{L}$ and $e_k$ for $k\in K$, we substitute these generators for $v_{\Tilde{N}}=(v_{\Tilde{L}}\ v_K)$ in \eqref{e:SKL-rep-final} and obtain that the set of $(K,L)$-eigenvectors is the max-min convex hull of the points described by~\eqref{e:points}. 
\end{proof}

For the purposes of computation note that the minimal solutions of~\eqref{e:finaleq}, which correspond to minimal coverings of $I_0$~\eqref{e:I0new} by unions of $C_j$~\eqref{e:Cjnew}, can be found using the methods described
in Elbassioni~\cite{Elb-08}. The set of all $(K,L)$-eigenvectors can be then efficiently described using Theorem~\ref{t:mainres} and Corollary~\ref{c:mainres}. 

In particular, the number of minimal solutions of $A'z'\oplus b'=\lambda \mathbf{1}$ is equal to the number of minimal coverings $\{\cup C_j\colon j\in W\}$ of $I_0$, which depends on entries of $A$ (particularly $(A_{KK})^*$) and given $\lambda$. However, the dependency of the points generating the set of $(K,L)$-eigenvectors as their max-min convex hull on these minimal coverings seems rather uncertain and can be eliminated if $A_{KK})^*_{\lambda}z_K^W$ is dominated by the sum of other terms in~\eqref{e:SKL-rep-final}. 

In the following example, there is just one possible minimal covering that produces the only minimal solution used in further computation.

%
%Thus,
%\eqref{e:SKL-rep-final} (together with~\eqref{e:SLSK}) provides an explicit expression for the set of
%all $(K,L)$ eigenvectors.
%}}
%
%\begin{equation}
%(z_{\Tilde{L}}^W \oplus M v_{\Tilde{L}})
%(z_{K}^W \oplus M v_{K})
%\end{equation}
%
\begin{example}\normalfont
\label{e:exfinal}
The following three-dimensional example shows how to find the eigenspace for given matrix $A$
and eigenvalue $\lambda$.

Take
\begin{equation*}
A=
\begin{pmatrix}
.1 & .5 & .7\\
0 & .4 & .8\\
.1 & .1 & .5
\end{pmatrix},
\ \lambda = .5
\end{equation*}
First, let us introduce the notation of the vector with interval entries. The vector where entries are of form $a \leq x_1 \leq b$ and $c \leq x_2 \leq d$ is denoted in further text as
\begin{equation*}
x=
\begin{pmatrix}
\langle a,b \rangle \\
\langle c,d \rangle
\end{pmatrix}.
\end{equation*}
As the $\lambda$-eigenspace is a union of background, principal and $(K,L)$ $\lambda$-eigenvectors, we are going to compute the solution for each individual case.

For the case of background $\lambda$-eigenvectors, first, we have to verify the existence of this eigenvector-type. From \eqref{e:backgr-nonempty} we see that this set is nonempty. According to \eqref{e:backgr}, background eigenvectors are all vectors of form
\begin{equation}
x=
\begin{pmatrix}
\langle .5,1 \rangle \\
\langle .5,1 \rangle \\
 .5
\end{pmatrix}
\end{equation}
Their set is the max-min convex hull of 
$(.5, .5, .5),$ $(1, .5, .5),$ 
$(.5, 1, .5)$.

Principal $\lambda$-eigenvectors are computed using Corollary \ref{c:principal}. We find that generators of max-min eigenspace for $x_i \leq \lambda$ for all $i$ are $(.1, .1, .1)$, $(.4, .4, .1)$ and $(.5, .5, .5)$.
Of these vectors, the first one is obviously redundant as $(.1, .1, .1)= .1\otimes  (.5, .5, .5)$. If we want to describe this eigenspace as a max-min convex hull then the zero vector $(0, 0, 0)$ is routinely added to the generating set.

When computing $(K,L)$ eigenvectors, first we have to determine particular $(K,L)$ partition. For partition $K=\{1,3\}$, $L=\{2\}$
we have
\begin{equation*}
A_{KK}=%A_{KK}^+=
\begin{pmatrix}
.1  & .7\\
.1 & .5
\end{pmatrix},
A_{KL}=
\begin{pmatrix}
.5\\
.1
\end{pmatrix},
A_{LK}=
\begin{pmatrix}
0  & .8
\end{pmatrix},
A_{LL}=
\begin{pmatrix}
.4
\end{pmatrix},
\end{equation*}
\begin{equation*}
\ (A_{KK})^*=
\begin{pmatrix}
1  & .7\\
.1 & 1
\end{pmatrix},
\ (A_{KK})_{\lambda}^*=
\begin{pmatrix}
.1  & .5\\
.1 & .5
\end{pmatrix}.
\end{equation*}
We are solving the system
\begin{align}
\label{Keqn1}
\begin{pmatrix}
.1  & .7\\
.1 & .5
\end{pmatrix}
\otimes
\begin{pmatrix}
x_1\\
x_3
\end{pmatrix}
\oplus
\begin{pmatrix}
.5\\
.1
\end{pmatrix}
\otimes
x_2
&=
\begin{pmatrix}
x_1\\
x_3
\end{pmatrix} \\
\label{Leqn1}
\begin{pmatrix}
0  & .8
\end{pmatrix}
\otimes
\begin{pmatrix}
x_1\\
x_3
\end{pmatrix}
\oplus
\begin{pmatrix}
.4
\end{pmatrix}
\otimes
x_2
&= \lambda
\end{align}
By verifying \eqref{Keqn-cond}, we find out that\eqref{Keqn-cond} holds for all values $x_2$, and thus a solution to \eqref{Keqn1} exists.

We can also express the sets $L_1=\emptyset, L_2=\{2 \}, L^{\prime}= \emptyset$ and $\Tilde{L}=\{2 \}$. As $L_2\neq \emptyset$, we need to find the solution set to \eqref{e:finaleq}:
\begin{align}
\label{Keqn2}
.5
\otimes
\begin{pmatrix}
0 & .8
\end{pmatrix}
\otimes
\begin{pmatrix}
1  & .7\\
.1 & 1
\end{pmatrix}
\otimes
\begin{pmatrix}
.5\\
.1
\end{pmatrix}
\otimes
1
\oplus  \hspace{147pt} \nonumber
 \\
\begin{pmatrix}
0 & .8
\end{pmatrix}
\otimes
\begin{pmatrix}
1  & .7\\
.1 & 1
\end{pmatrix}
\otimes
\begin{pmatrix}
.5\\
.1
\end{pmatrix}
\otimes
1
\otimes
z_2
\oplus %\\
\begin{pmatrix}
0 & .8
\end{pmatrix}
\otimes
\begin{pmatrix}
.1  & .5\\
.1 & .5
\end{pmatrix}
\otimes
\begin{pmatrix}
z_1\\
z_3
\end{pmatrix}
=
.5,
\end{align}
which is the same as
\begin{equation}
\label{Leqn2}
.1
\oplus
\begin{pmatrix}
.1 & .1 & .5
\end{pmatrix}
\otimes
\begin{pmatrix}
z_1\\
z_2\\
z_3
\end{pmatrix}
= .5
\end{equation}
The only minimal (and hence the least) solution to this system is $z^{\{3\}}=(0,0,.5)$, so the
solution set is $\{z\colon z^{\{3\}}\leq z\leq 1\}$.
%It corresponds to
%$$
%\Lambda^{\{3\}}=
%\begin{pmatrix}
%1 & 0 & 0\\
%0 & 1 & 0\\
%.5 & .5 & 1
%\end{pmatrix}.
%$$
%
We
can also express $
z_{\tilde{L}}^{\{3\}}=z_L^{\{3\}}=(0)$, $z_K^{\{3\}}=(0,.5)$.

Following Theorem~\ref{t:mainres} and using equations~\eqref{e:SKL-rep-final} we obtain
\begin{equation}
\label{e:final}
\begin{split}
 x_2&=.5\oplus v_2,\\
\begin{pmatrix} x_1 \\ x_3 \end{pmatrix}&=
.5\begin{pmatrix} 1 & .7\\ .1 & 1\end{pmatrix}\begin{pmatrix}.5\\ .1\end{pmatrix}\oplus
\begin{pmatrix} 1 & .7\\ .1 & 1\end{pmatrix}\begin{pmatrix}.5\\ .1\end{pmatrix}v_2\\
&\oplus \begin{pmatrix} .1 & .5\\ .1 & .5\end{pmatrix}\Biggl(\begin{pmatrix} 0\\ .5\end{pmatrix}
\oplus \begin{pmatrix} v_1\\ v_3\end{pmatrix}\Biggl)\\
&=\begin{pmatrix} .5\\ .1\end{pmatrix}(1\oplus v_2)\oplus \begin{pmatrix} .5\\ .5\end{pmatrix}
=
\begin{pmatrix} .5\\ .5\end{pmatrix},
\end{split}
\end{equation}
which can be written as
\begin{equation}
x=
\begin{pmatrix}
.5 \\
\langle .5,1 \rangle \\
 .5
\end{pmatrix}
\end{equation}
For the partition $K=\{2,3\}$, $L=\{1\}$ we are solving the system
\begin{align}
\label{Keqn3}
\begin{pmatrix}
.4  & .8\\
.1 & .5
\end{pmatrix}
\otimes
\begin{pmatrix}
x_2\\
x_3
\end{pmatrix}
\oplus
\begin{pmatrix}
0\\
.1
\end{pmatrix}
\otimes
x_1
&=
\begin{pmatrix}
x_2\\
x_3
\end{pmatrix} \\
\label{Leqn3}
\begin{pmatrix}
.5 & .7
\end{pmatrix}
\otimes
\begin{pmatrix}
x_2\\
x_3
\end{pmatrix}
\oplus
\begin{pmatrix}
.1
\end{pmatrix}
\otimes
x_1
&= \lambda
\end{align}
Similarly to previous procedures we compute eigenvectors for this partition
\begin{equation}
x=
\begin{pmatrix}
\langle .5,1 \rangle \\
.5 \\
 .5
\end{pmatrix}
\end{equation}
%
%Considering any other $(K,L)$ partition we will find out that the solution set is empty. The final solution set is then a union of computed parts: background, principal and $(K,L)$ eigenvectors for $L=\{1\}$ and $L=\{2\}$, see Figure \ref{f:exfinal}. Note however that all $(K,L)$-eigenvectors are also background eigenvectors
%in this case. This is different from Example~\ref{ex:1} where we have $(K,L)$-eigenvectors are
%neither background nor principal.

Thus we see that $(K,L)$-eigenvectors for $L=\{1\}$ and $L=\{2\}$ are background eigenvectors. By similar routine calculations we can check that the same is true for $(K,L)$-eigenvectors for any non-empty $L$, and this is different from Example 3.1, where  we have $(K,L)$-eigenvectors that are
neither background nor principal. Note that in Example 3.3 any set of $(K,L)$-eigenvectors is non-empty, as $(.5,.5,.5)$ is a 
$\lambda$-eigenvector, which is shared between all of them.

\begin{figure}
  \centering
  \includegraphics[width=250pt]{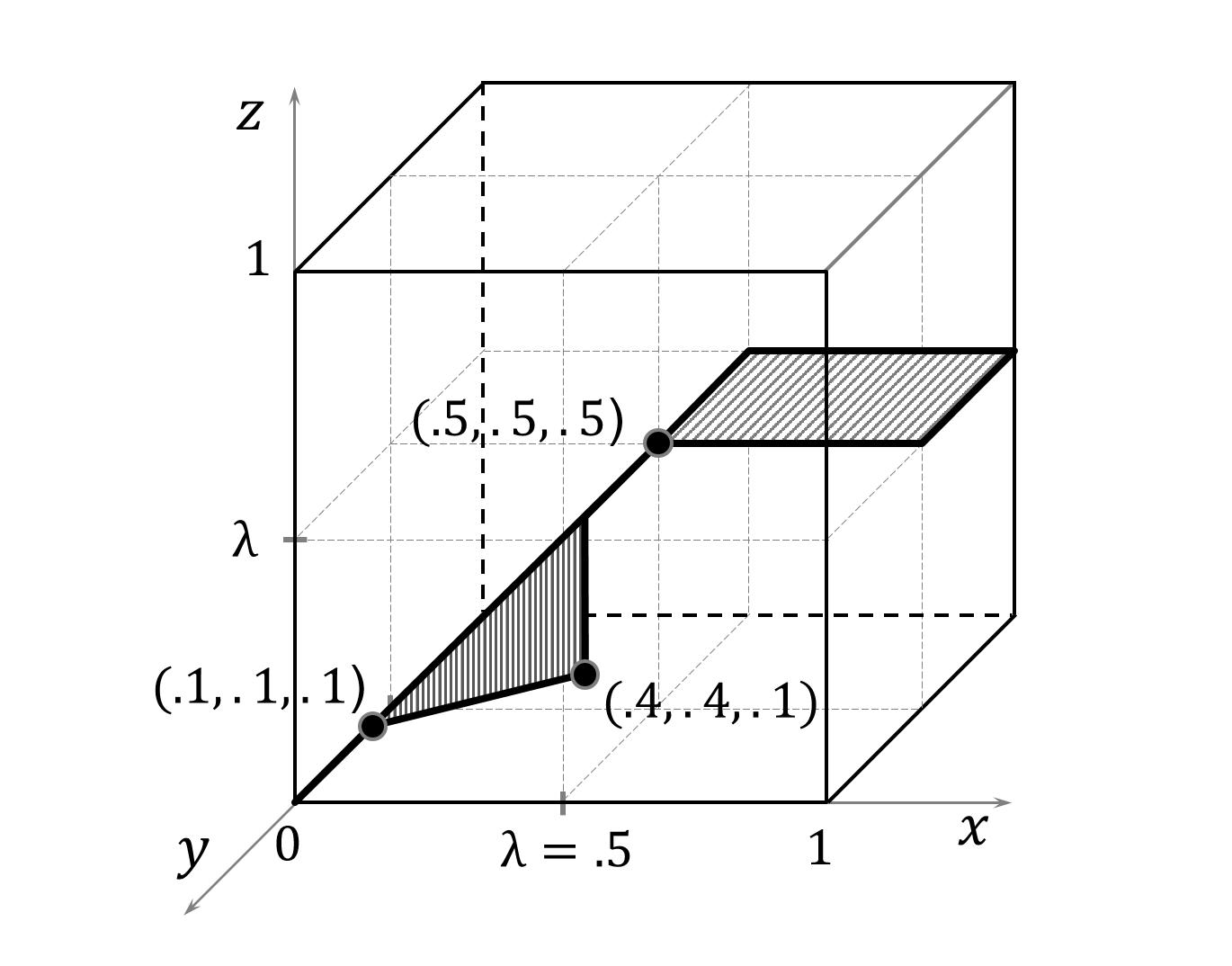}
  \caption{Eigenspace for $A$ and $\lambda$ of Example~\ref{e:exfinal} \label{f:exfinal} }
\end{figure}

The whole $\lambda$-eigenspace, shown on Figure~\ref{f:exfinal}, is thus the union of two parts: principal eigenvectors and background eigenvectors. This eigenspace (as a union of max-min convex hulls which is itself max-min convex) is the max-min convex hull of: $(.5,.5,.5)$, $(1,.5,.5)$, $(.5,1,.5)$, $(.4,.4,.1)$ and $(0,0,0)$. Removing the zero vector we obtain a generating set of the whole $\lambda$-eigenspace, and next we notice that $(.5,.5,.5)$ is redundant, since for example $(.5,.5,.5)=.5\otimes (1,.5,.5)$. Thus the $\lambda$-eigenspace is generated by $(1,.5,.5)$, $(.5, 1, .5)$ and $(.4,.4,.1)$.
\end{example}

\section{Discussion}

Theorem~\ref{t:mainres} and Corollary~\ref{c:mainres} provide us with a description of the sets of $(K,L)$-eigenvectors associated with eigenvalue $\lambda$ and, in particular, give us a finite set of points such that the set of $(K,L)$ max-min eigenvectors appears as the max-min convex hull of them. Collecting such points for all sets of $(K,L)$-eigenvectors we obtain a finite set of generators for the whole $\lambda$-eigenspace. 

There are some problems with this description. Clearly, there are exponentially many partitions $K\cup L=\{1,\ldots, n\}$, and to implement the above described method for finding the generators of $\lambda$-eigenspace one should find a way to quickly eliminate many partitions that are redundant: in two of the three examples that we considered, only the trivial partitions, corresponding to principal and background eigenvectors, were important. The procedure also raises the question about the number of points that generate the $\lambda$-eigenspace. From Example 3.2, it is clear that the number of such points can exceed the dimension, as the eigenspace in this example is generated by $(.4,.2)$, $(.5, 1),$ $(1,.5)$ (e.g., consider it as a ``max-min quadrangle'' between these points and the zero vector, and use the forms of max-min segments given in~\cite{NS-15}), with none of these points being redundant. However, it is not known how quickly the minimal number of such generators can grow with matrix dimension. 

Similar questions can be asked also about the set of $(K,L)$-eigenvectors. The points that generate this set as their max-min convex hull come from certain minimal coverings, whose number may also grow quickly with matrix dimension, unless $L_2=\emptyset$ as in Corollary~\ref{solvability_5}. However, we have seen that the terms coming from these minimal coverings are often dominated by other terms, thus decreasing the number of points making up the max-min convex hull.           

There is a perspective to develop an application of max-min $\lambda$-eigenspaces in medical diagnostics following the ideas of Sanchez~\cite{San-78}. Also, there are still many unresolved questions in the geometry over max-min semiring:  about the generating sets of max-min eigenspaces and more general max-min linear spaces, and more generally, about max-min polytopes.

\section{Acknowledgements}

We thank our anonymous referees for their careful reading and constructive criticism, which 
helped to improve the paper.

%\section*{References}

%\bibliographystyle{alpha}
%\bibliography{mariecurie}

\begin{thebibliography}{9}

\bibitem{BCOQ}
F.L.~Baccelli, G.~Cohen, G.J.~Olsder and J.P.~Quadrat.
{\em Synchronization and Linearity.} John Wiley and Sons, New York, 1992.
%Available online: %\url{http://www.rocq.inria.fr/metalau/cohen/SED/book-online.html}.

\bibitem{But:10}
P.~Butkovi\v{c}. {\em Max-Linear Systems: Theory and Algorithms.} Springer, London, 2010.

\bibitem{BSS-12}
P.~Butkovi\v{c}, H.~Schneider, and S.~Sergeev.
\newblock Z-matrix equations in max-algebra, nonnegative linear algebra and
  other semirings.
\newblock {\em Linear and Multilinear Alg.}, {\bf 60}(10):1191--1210, 2012.
\newblock %E-print \arxiv{1110.4564}.


\bibitem{CKR:84}
Z.-Q. Cao, K.H. Kim, and F.W. Roush.
\newblock {\em Incline {A}lgebra and its {A}pplications}.
\newblock Chichester, 1984.

\bibitem{Car-71}
B.A. Carr\'{e}.
\newblock An algebra for network routing problems.
\newblock {\em J. of the Inst. of Maths. and Applics}, 7, 1971.

\bibitem{cechlarova1992}
K.~Cechl\'arov\'a.
\newblock Eigenvectors in bottleneck algebra.
\newblock {\em Linear Algebra and its Applications}, 175:63-73, 1992.

%Paper
\bibitem{cechlarova2000note}
K.~Cechl\'arov\'a.
\newblock A note on unsolvable systems of max--min (fuzzy) equations.
\newblock {\em Linear Algebra and its Applications}, 310:123-128, 2000.

\bibitem{CC-95}
R.~A. Cuninghame-Green, K.~Cechl\'arov\'a.
\newblock Residuation in fuzzy algebra and some applications.
\newblock {\em Fussy Sets and Systems}, 71(2):227-239, 1995.

\bibitem{Elb-08}
K.~Elbassioni.
\newblock A note on systems with max-min and max-product constraints.
\newblock{\em Fuzzy Sets and Systems}, {\bf 159}, 2008, 2272-2277.

\bibitem{G-02}
M. Gavalec.
\newblock Monotone eigenspace structure in max-min algebra,
\newblock {\em Linear Algebra and its Applications}, {\bf 345}, 2002, 149-167.

\bibitem{GN-17}
M. Gavalec and Z. N\v{e}mcov\'{a}.
\newblock Steady states of max-{\L}ukasiewicz fuzzy systems,
\newblock {\em Fuzzy Sets and Systems},
{\bf 325}, 2017, 58-68.

\bibitem{GNS-15}
M. Gavalec, Z. N\v{e}mcov\'{a}, S. Sergeev.
\newblock Tropical linear algebra with the {\L}ukasiewicz T-norm,
\newblock {\em Fuzzy Sets and Systems}, {\bf 276}, 2015, 131-148.

\bibitem{GPP-19}
M. Gavalec, J. Plavka, D. Ponce.
\newblock Strong tolerance of interval eigenvectors in fuzzy algebra, \newblock {\em Fuzzy Sets and Systems}, {\bf 369}, 2019, 145-156.

\bibitem{GPT-14}
M. Gavalec, J. Plavka, H. Tom\'{a}\v{s}kov\'{a}.
\newblock Interval eigenproblem in max-–min algebra,
\newblock {\em Linear Algebra and its Applications}, {\bf 440}, 2014, 24-33.

\bibitem{Gon-76}
M.~Gondran.
\newblock Valeurs propres et vecteurs propres en
classification hi\'erarchique,
\newblock{\em R.A.I.R.O. Informatique Th\'eorique},
{\bf 10}(3), 1976, 39-46.

\bibitem{GM-78}
M.~Gondran and M.~Minoux.
\newblock Valeurs propres et vecteurs propres en
th\'eorie des graphes, in
\newblock{\em Colloques Internationaux,}
C.N.R.S., Paris 1978, pp. 181-183.

\bibitem{GM-07}
M.~Gondran and M.~Minoux.
\newblock Dio{\"{\i}}ds and Semirings: Links to fuzzy sets and other applications.
\newblock {\em Fuzzy Sets and Systems}, {\bf 158}, 2007, 1273--1294.

\bibitem{GM:08}
M.~Gondran and M.~Minoux.
\newblock {\em Graphs, Dioids and Semirings: New Applications and Algorithms}.
\newblock Springer, 2008.

%\bibitem{HL-04}
%S.-C. Han and H.-X. Li.
%\newblock Invertible incline matrices and {C}ramer's rule over inclines.
%\newblock {\em Linear Alg. Appl.}, 389:121--138, 2004.

\bibitem{KMP:00}
E.P. Klement, R.~Mesiar, and E.~Pap.
\newblock {\em Triangular Norms}.
\newblock Kluwer Academic Publ., Dordrecht, 2000.

\bibitem{Kri-06}
N.~K. Krivulin.
\newblock On solution of generalized linear vector equations in idempotent
  algebra.
\newblock {\em Vestnik St.-Petersburg Univ. Math.}, {\bf 39}, 2006, 16--26.

\bibitem{Kri:09}
N.~K. Krivulin.
\newblock {\em Methods of idempotent algebra in the problems of modeling and
  analysis of complex systems}.
\newblock St.-Petersburg Univ. Publ., 2009.
\newblock (in Russian)

\bibitem{LM-98}
G.~L. Litvinov and V.~P. Maslov.
\newblock  The correspondence principle for idempotent calculus and some
  computer applications.
\newblock In J.~Gunawardena (ed.) {\em Idempotency}, 
  Cambridge Univ. Press, 1998, pages 420-443.
%\newblock E-print %\arxiv{math/0101021}.
\bibitem{NS-14}
 V. Nitica and S. Sergeev. 
 \newblock Tropical convexity over max-min semiring. 
 \newblock In G.~L. Litvinov and S.~N. Sergeev (eds.) {\em Tropical and Idempotent Mathematics and Applications}, vol. {\bf 616} of Cont. Math. series, American Mathematical Society, 2014, pages 241-260.

\bibitem{NS-15} 
V. Nitica and S. Sergeev.
\newblock On the dimension of max-min convex sets.
\newblock {\em Fuzzy Sets and Systems},
{\bf 271}, 2015, 88-101.

\bibitem{NBH-06} 
H.~Nobuhara, B.~Bede and K.~Hirota.
\newblock On various eigen fuzzy sets and their application to image reconstruction. 
\newblock {\em Information Sciences} {\bf 176}, 2006, 2988-3010.

\bibitem{Rak:07}
E.~Rakus-Andersson, {\em {F}uzzy and {R}ough {T}echniques in {M}edical {D}iagnosis and {M}edication}, vol. 212 of StudFuzz series, Springer, Berlin, 2007.  

\bibitem{San-78}
E.~Sanchez. 
\newblock Resolution of eigen fuzzy sets 
equations. 
\newblock {\em Fuzzy Sets and Systems} {\bf 1}(1), 1978, 69--74. 


\bibitem{San-81}
E.~Sanchez.
\newblock Eigen fuzzy sets and fuzzy relations.
\newblock {\em J. of Math. Analysis and Appl.} {\bf 81}, 1981, 399--421.

%\bibitem{Tan-11}
%Y.-J. Tan.
%\newblock On generalized fuzzy matrices with periods.
%\newblock {\em Fuzzy Sets and Systems}, 172(1):87--103, 2011.
\end{thebibliography}

\end{document}